\documentclass[11pt]{article}
\usepackage{mathrsfs}
\usepackage{amsmath,amssymb,amsthm,bm}
\usepackage{geometry}
\usepackage[colorlinks,linkcolor=blue]{hyperref}

\def\proclaim#1{\par \smallskip\noindent {\bf #1}\bgroup\it\ }
\def\endproclaim{\egroup\par\smallskip}

\newtheorem{lemma}{Lemma}[section]
\newtheorem{theorem}{Theorem}[section]
\newtheorem{proposition}{Proposition}[section]
\newtheorem{corollary}{Corollary}[section]

\def\pr{\textsf{P}}
\def\ep{\textsf{E}}
\def\Cov{\textsf{Cov}}
\def\Var{\textsf{Var}}

\begin{document}

\title{\bf Strong Invariance Principle for Random Dirichlet Series}

\author{Li-Xin ZHANG\footnote{Address: School of Statistics and Data Sciences,
Zhejiang Gongshang University, 310018, China.
Email: stazlx@mail.zjgsu.edu.cn}\\
Zhejiang Gongshang University, China}

\date{}
\maketitle

\begin{abstract}
It is shown that the random prime Dirichlet series
$$
X(s)=\sum_{p\in \mathcal{P}}\frac{\xi_p}{p^{s+1/2}}
$$
can be almost surely approximated by a Brownian motion, where
$\{\xi_p\}_{p\ge 1}$ are independent and identically distributed random
variables with zero mean and finite variance, and $\mathcal{P}$ is the set
of prime numbers. As a consequence, the functional law of the iterated
logarithm, limit laws of the iterated logarithm, and Chung's law of the
iterated logarithm for $X(s)$ are obtained. It is also shown that the random
Dirichlet series
$$
Y(s)=\sum_{k=2}^{\infty}\frac{(\log k)^{\alpha}}{k^{s+1/2}}\xi_k
$$
with $\alpha=-1/2$ admits the same approximation as $X(s)$, and a direct
relationship between $X(s)$ and $Y(s)$ with $\alpha=-1/2$ is established.
The strong invariance principle for $Y(s)$ with $\alpha>-1/2$ is also
obtained.


\end{abstract}

\section{Main results}
\setcounter{equation}{0}

In this paper, let $\{\xi,\xi_p;\,p\ge 1\}$ be a sequence of independent,
zero-mean random variables with finite variance $\sigma^2$, defined on a
probability space $(\Omega,\mathcal{F},\pr)$, and let $\mathcal{P}$ denote the
set of prime numbers. Write $\log x=\ln(x\vee e)$. Consider the random prime
Dirichlet series
\begin{equation}\label{eq:0}
X(s)=\sum_{p\in \mathcal{P}}\frac{\xi_p}{p^{s+1/2}},\qquad s>0,
\end{equation}
and the random Dirichlet series
\begin{equation}\label{eq:1}
Y(s)=Y_{\alpha}(s)=\sum_{k=1}^{\infty}\frac{(\log k)^{\alpha}}{k^{s+1/2}}\xi_k,
\qquad s>0.
\end{equation}

Central limit theorems (CLTs) and laws of the iterated logarithm (LILs) for
the random Dirichlet series $Y(s)$ have attracted considerable attention in
recent years. Aymone et al. \cite{AFM2020} established the LIL for Dirichlet
series with $\pr(\xi=\pm 1)=1/2$. Buraczewski et al. \cite{BDIM2023}
considered the case $\alpha>-1/2$ under the assumption of a finite second
moment. They obtained a functional central limit theorem (FCLT) for the
random Dirichlet series in a space of analytic functions and a law of the
iterated logarithm (LIL) in the real case. Gao and Xia \cite{GaoXia2026}
obtained the functional large deviation principle and functional LILs for the
case $\alpha>-1/2$. Iksanov and Kostohryz \cite{IK2025} studied the critical
case $\alpha=-1/2$ and obtained an FCLT and an LIL. The critical case is
closely related to some problems in probabilistic number theory
\cite{GeisHiary2025}. Iksanov and Kostohryz \cite{IK2025} conjectured that
the same FCLT and LIL for the critical case of $Y(s)$ also hold for the
random prime Dirichlet series $X(s)$; these are stated as two conjectures.

Recently, Dong and Iksanov \cite{DI2025} obtained an FCLT and an LIL for the
random prime Dirichlet series $X(s)$, thereby proving the conjectures of
Iksanov and Kostohryz \cite{IK2025}. As a consequence, they obtained an FCLT
and an LIL for the Dirichlet series of Rademacher random multiplicative
functions. Iksanov and Kostohryz \cite{IK2025} mentioned that the proof of
the LIL is quite technical and that they were unable to find a less technical
argument based on the strong approximation. The purpose of this paper is to
establish a strong invariance principle for the random prime Dirichlet series
$X(s)$ as well as for the random Dirichlet series $Y(s)$.

We will show that both $X(s)$ and $Y(s)$ (in the critical case $\alpha=-1/2$)
can be approximated by a Brownian motion. Using this approximation, we recover
the FCLTs of Dong and Iksanov \cite{DI2025} and Iksanov and Kostohryz
\cite{IK2025}, and establish the functional LIL, Chung's LIL, as well as the
limit laws of the iterated logarithm (as opposed to laws for the limit
superior). When $\alpha>-1/2$, we will show that $Y(s)$ can be approximated
by the Laplace transform of a power-weighted Brownian stochastic integral,
and recover the FCLT of Buraczewski et al. \cite{BDIM2023} and the functional
LIL of Gao and Xia \cite{GaoXia2026}. Based on the approximation, the precise
Chung LIL is stated as a conjecture.

The following is our main theorem on the strong invariance principle for the
random prime Dirichlet series $X(s)$.

\begin{theorem}\label{th:1}
There exists a standard Brownian motion $B(t)$ such that
\begin{equation}
\label{eq:th:1.1}
X(s)-\sigma B(\log (1/s))
=o\Big(\sqrt{\log(1/s)\log\log\log(1/s)}\Big)\quad a.s.
\text{ as }s\to 0,
\end{equation}
and
\begin{equation}
\label{eq:th:1.2}
\sup_{s\le t\le 1}\Big|X(t)-\sigma B(\log (1/t))\Big|
=o\Big(\sqrt{\log (1/s)}\Big)\quad \text{in }L_2
\text{ as }s\to 0.
\end{equation}
Furthermore, if
\begin{equation}\label{eq:th:1.3}
\ep[\xi^2(\log |\xi|)^{1+\epsilon}]<\infty
\text{ for some }\epsilon>0,
\end{equation}
then
\begin{equation}
\label{eq:th:1.4}
X(s)-\sigma B(\log (1/s))
=O\Big(\sqrt{\log\log (1/s)}\Big)\quad a.s.
\text{ as }s\to 0.
\end{equation}
Here and in the sequel, for two processes $X$ and $Y$, writing
$X(s)-Y(s)=O(b_s)$ a.s. (in probability, in $L_2$, respectively) means
that, possibly in an enlarged probability space, there exist versions
$\widetilde X$ of $X$ and $\widetilde Y$ of $Y$ such that
$\widetilde X(s)-\widetilde Y(s)=O(b_s)$ a.s. (in probability, in $L_2$,
respectively).
\end{theorem}

The following FCLT and functional LIL follow from Theorem \ref{th:1} and the
properties of a Brownian motion directly.

\begin{corollary}\label{cor:1}
Consider $X(s^t),t\ge 0$ as elements in $C[0,\infty)$ equipped with the
topology of local uniform convergence. We have, as $s\to 0$,
\begin{align}\label{eq:cor:1-1}
\left(\frac{X(s^t)}{\sqrt{\log (1/s)}}\right)_{t\ge 0}
\implies \Big(\sigma B(t)\Big)_{t\ge 0},
\end{align}
\begin{align}\label{eq:cor:1-2}
\operatorname{Clust}\left\{
\left(\frac{X(s^t)}
{\sqrt{2\log (1/s)\log\log\log (1/s)}}\right)_{t\ge 0};
\ 0<s\le 1
\right\}
=\mathscr{K}_B,
\end{align}
where
\begin{align}\label{eq:cor:1-3}
\mathscr{K}_B=\left\{
\sigma f;\ f(0)=0,\ f\text{ is absolutely continuous},\
\int_0^{\infty}[f^{\prime}(t)]^2\,dt\le 1
\right\}.
\end{align}
\end{corollary}

By the limit laws of the iterated logarithm (cf. Chen \cite{ChenX2015},
Robbins and Siegmund \cite{RS1972}) as
$$
\lim_{T\to \infty}\left(\sup_{1\le t\le T}\frac{W(t)}{\sqrt{t}}
-\sqrt{2\log\log T}\right)
=\lim_{T\to \infty}
\left(\sup_{0\le u \le T}e^{-u/2}W(e^u)-\sqrt{2\log T}\right)=1
\quad a.s.,
$$
we have the following corollary.

\begin{corollary}\label{cor:2}
We have
\begin{align}\label{eq:cor:2-1}
\lim_{s\to 0}\frac{1}{\sqrt{2\log\log\log (1/s)}}
\sup_{s\le t\le 1}\frac{X(t)}{\sqrt{\log (1/t)}}
=\sigma\quad a.s.
\end{align}
Furthermore, if \eqref{eq:th:1.3} holds, then
\begin{align}\label{eq:cor:2-2}
\lim_{s\to 0}
\left(
\sup_{s\le t\le 1}\frac{X(t)}{\sqrt{\log (1/t)}}
-\sigma\sqrt{2\log\log\log (1/s)}
\right)=0\quad a.s.
\end{align}
\end{corollary}

By Chung's law of the iterated logarithm of a Brownian motion (cf.
Cs\"org\H{o} and R\'ev\'esz \cite{CR81}), we have the following same type of
law for $X(s)$.

\begin{corollary}\label{cor:3}
Suppose \eqref{eq:th:1.3}. Then
\begin{align}\label{eq:cor:3.1}
\liminf_{s\to 0}\sqrt{\frac{\log\log\log (1/s)}{\log (1/s)}}
\sup_{s\le t\le 1}|X(t)|
=\sigma \frac{\pi}{\sqrt{8}}\quad a.s.
\end{align}
\end{corollary}

We do not know whether \eqref{eq:th:1.3} can be removed or not in Corollary
\ref{cor:3}.

\bigskip
For random Dirichlet series $Y(s)$ with $\alpha=-1/2$, we have the same
properties.

\begin{theorem}\label{th:2}
For $Y(s)$ with $\alpha=-1/2$, there exists a standard Brownian motion
$B(t)$ such that
\begin{equation}
\label{eq:th:2.1}
Y(s)-\sigma B(\log (1/s))
=o\Big(\sqrt{\log(1/s)\log\log\log(1/s)}\Big)\quad a.s.
\text{ as }s\to 0,
\end{equation}
and
\begin{equation}
\label{eq:th:2.2}
\sup_{s\le t\le 1}\Big|Y(t)-\sigma B(\log (1/t))\Big|
=o\Big(\sqrt{\log (1/s)}\Big)\quad \text{in }L_2
\text{ as }s\to 0.
\end{equation}
Furthermore,
\begin{equation}
\label{eq:th:2.3}
Y(s)-\sigma B(\log (1/s))
=O\Big(\sqrt{\log\log (1/s)}\Big)\quad a.s.
\text{ as }s\to 0,
\end{equation}
if \eqref{eq:th:1.3} holds.
\end{theorem}

Corollaries \ref{cor:1}, \ref{cor:2} and \ref{cor:3} remain true when
$X(s)$ is replaced by $Y(s)$ with $\alpha=-1/2$. \eqref{eq:cor:1-1} recovers
the FCLTs of Dong and Iksanov \cite{DI2025} and Iksanov and Kostohryz
\cite{IK2025}, respectively. \eqref{eq:cor:1-2}, \eqref{eq:cor:2-1},
\eqref{eq:cor:2-2} and \eqref{eq:cor:3.1} are new.

\bigskip
The following theorem is the strong invariance principle for the random
Dirichlet series $Y(s)$ with $\alpha>-1/2$.

\begin{theorem}\label{th:3}
Let $\alpha>-1/2$. Then there exists a standard Brownian motion $B(t)$ such
that
\begin{equation}
\label{eq:th:3.1}
s^{1/2+\alpha}\left(Y(s)-\sigma \int_0^{\infty} y^{\alpha}e^{-sy}\,dB(y)\right)
=o\Big(\sqrt{\log\log(1/s)}\Big)\quad a.s.
\text{ as }s\to 0,
\end{equation}
and
\begin{equation}
\label{eq:th:3.2}
s^{1/2+\alpha}\sup_{s\le t\le 1}
\Big|Y(t)-\sigma \int_0^{\infty} y^{\alpha}e^{-ty}\,dB(y)\Big|
\to 0\quad \text{in }L_2
\text{ as }s\to 0.
\end{equation}
Furthermore, if \eqref{eq:th:1.3} holds, then for
$\tau<\frac{\epsilon}{2}\wedge(\alpha+\frac{1}{2})$,
\begin{equation}
\label{eq:th:3.3}
s^{1/2+\alpha}\left(Y(s)-\sigma \int_0^{\infty} y^{\alpha}e^{-sy}\,dB(y)\right)
=o\big(s^{\tau}\big)\quad a.s.
\text{ as }s\to 0.
\end{equation}
\end{theorem}

The process
$\int_0^{\infty} e^{-sy} y^{\alpha}\,dB(y)$ is the Laplace transform of the
power-weighted Brownian stochastic integral
$\int_0^{\cdot} y^{\alpha}\,dB(y)$. The normalized process $G(s)=\frac{2^{2\alpha+1}}{\sqrt{\Gamma(2\alpha+1)}}s^{\alpha+1/2}
\int_0^{\infty} e^{-sy} y^{\alpha}\,dB(y)$ has the covariance function as
$$\ep[G(s)G(t)]=2^{2\alpha+1}\left(\frac{st}{(s+t)^2}\right)^{2\alpha+1}. $$
This class of processes was studied by Aurzada et al. \cite{Aurzada2013}
and is related to the diffusion equation
$$
\frac{\partial}{\partial t}\phi(\bm x,t)=\nabla^2\phi(\bm x,t),
$$
with initial condition
$\phi(\bm x,0)=\psi(\bm x)$,
where $\psi(\bm x)$ is a mean-zero Gaussian random field with
delta-function covariance
$\ep[\psi(\bm x)\psi(\bm y)]=\delta(\bm x-\bm y)$.

By the strong invariance principle, we recover the following functional CLT
and functional LIL.

\begin{corollary}\label{cor:4}
Let $\alpha>-1/2$. Consider $X(st),t>0$, as elements in $C(0,\infty)$
equipped with the topology of local uniform convergence. We have, as
$s\to 0$,
\begin{equation}\label{eq:cor:4.1}
\left(s^{\alpha+1/2} Y(st)\right)_{t> 0}
\implies
\left(\sigma \int_0^{\infty} y^{\alpha}e^{-ty}\,dB(y)\right)_{t> 0},
\end{equation}
\begin{equation}\label{eq:cor:4.2}
\operatorname{Clust}\left\{
\left(\frac{s^{\alpha+1/2} Y(st)}{\sqrt{2\log\log (1/s)}}\right)_{t> 0};
\ 0<s\le 1
\right\}
=\mathscr{K}_{\alpha},
\end{equation}
where
\begin{equation}\label{eq:cor:4.3}
\mathscr{K}_{\alpha}=\left\{
\sigma g;\ g(t)=\int_0^{\infty}y^{\alpha}e^{-ty}\,df(y),\ f\in K_B
\right\}.
\end{equation}
In particular,
\begin{equation}\label{eq:cor:4.4}
s^{\alpha+1/2} Y(s)\overset{d}\to
N\left(0,\sigma^2\frac{\Gamma(2\alpha+1)}{2^{2\alpha+1}}\right),
\end{equation}
\begin{equation}\label{eq:cor:4.5}
\operatorname{Clust}\left\{
\frac{s^{\alpha+1/2} Y(s)}{\sqrt{2\log\log (1/s)}};\ 0<s\le 1
\right\}
=
\left[-\sigma \sqrt{\frac{\Gamma(2\alpha+1)}{2^{2\alpha+1}}},
\sigma \sqrt{\frac{\Gamma(2\alpha+1)}{2^{2\alpha+1}}}\right]
\quad a.s.
\end{equation}
\end{corollary}

\eqref{eq:cor:4.2} recovers the FCLT of Buraczewski et al. \cite{BDIM2023}.
The expression for the limit point given in \eqref{eq:cor:4.3} is closer to
the form of the limit process given in \eqref{eq:cor:4.2} than the expression
given by Gao and Xia \cite{GaoXia2026}.

\begin{corollary}\label{cor:5}
We have
\begin{equation}\label{eq:cor:5.1}
\lim_{s\to 0}\frac{1}{\sqrt{2\log\log (1/s)}}
\sup_{s\le t\le 1} t^{\alpha+1/2} Y(t)
=\sigma\sqrt{\frac{\Gamma(2\alpha+1)}{2^{2\alpha+1}}}\quad a.s.
\end{equation}
and
\begin{equation}\label{eq:cor:5.2}
\lim_{s\to 0}\left(
\sup_{s\le t\le 1} t^{\alpha+1/2} Y(t)
-\sigma\sqrt{\frac{\Gamma(2\alpha+1)}{2^{2\alpha+1}}}
\sqrt{2\log\log (1/s)}
\right)=0\quad a.s.,
\end{equation}
if \eqref{eq:th:1.3} is satisfied.
\end{corollary}

The main idea of the proof of the strong invariance principle is a two-stage
approximation approach. First, we replace the random variables in the series
by normal random variables, as in Zhang \cite{Zhang1997} for geometrically
weighted random series, so that the series becomes a Gaussian process. Then
we approximate this Gaussian process by a Brownian motion or by a Laplace
transform of a power-weighted Brownian stochastic integral. The proof of
Theorem \ref{th:2} for the critical random Dirichlet series $Y(s)$
($\alpha=-1/2$) will be given in Section \ref{sec:critical}. Before that, in
Section \ref{sec:Gauss} we consider the special case in which the random
variables are normal, and show that, as a Gaussian process, $Y(s)$ can be
approximated by a Brownian motion when $\alpha=-1/2$, and by the Laplace
transform of the power-weighted Brownian stochastic integral when
$\alpha>-1/2$. In Section \ref{sec:prime}, Theorem \ref{th:1} for the random
prime Dirichlet series $X(s)$ is shown by establishing a direct relationship
between $X(s)$ and $Y(s)$ with $\alpha=-1/2$. In Section
\ref{sec:noncritical}, we consider $Y(s)$ with $\alpha>-1/2$ (the
noncritical case). In the last section, we give two remarks on the
noncritical case and give a conjecture on the precise Chung LIL.

\section{Dirichlet series of normal random variables}\label{sec:Gauss}
\setcounter{equation}{0}

In this section, we consider the special case that $\{\xi_k;k\ge 1\}$ are
normal random variables. Let $\{\eta_k;k\ge 1\}$ be independent standard
normal $N(0,1)$ random variables. Let
$$
Z(s)=Z_{\alpha}(s)=\sum_{k=1}^{\infty}
\frac{(\log k)^{\alpha}}{k^{s+1/2}}\eta_k,\qquad \alpha\ge -1/2.
$$

\begin{proposition}\label{prop:2-1}
Let $\alpha=-1/2$. Then there exists a standard Brownian motion $B(t)$ such
that
\begin{equation}\label{eq:prop:2-1.2}
Z_{\alpha}(s)-B(\log (1/s))
=O\Big(\sqrt{\log\log (1/s)}\Big)\quad a.s.
\text{ as }s\to 0.
\end{equation}
\end{proposition}

\begin{proposition}\label{prop:2-2}
Let $\alpha>-1/2$. Then there exists a standard Brownian motion $B(t)$ such
that
\begin{equation}\label{eq:prop:2-2-1}
Z_{\alpha}(s)-\int_0^{\infty} y^{\alpha}e^{-sy}\,dB(y)
=O\Big(s^{-\alpha\vee 0}\sqrt{\log (1/s)}\Big)\quad a.s.
\text{ as }s\to 0.
\end{equation}
\end{proposition}

To prove Proposition \ref{prop:2-1}, we need a lemma on the variances and
covariances of the increments of $Z_{\alpha}(s)$.

\begin{lemma}\label{lem:4-1} Let $\alpha=-1/2$.
\begin{description}
\item[\rm (i)]
\begin{equation}\label{eq:lem:4-1:1}
\ep\bigl[Z_{\alpha}^2(s)\bigr]
=\sum_{k=1}^{\infty}\frac{1}{k^{1+2s}(\log k)}
=\log(1/s)+O(1)\qquad\text{as }s\to 0,
\end{equation}
and there exist constants $C_1,C_2,C_3>0$ such that
\begin{equation}\label{eq:lem:4-1:2}
C_1t-C_3
\le
\ep\left[\left(Z_{\alpha}(e^{-s-t})-Z_{\alpha}(e^{-s})\right)^2\right]
\le C_2t
\end{equation}
for all $t>0$ and $s>0$.
\item[\rm (ii)]
Fix $t_0>0$, and set
$$
Y_n=Z_{\alpha}(e^{-t_0(n-1)})-Z_{\alpha}(e^{-t_0n}),\qquad n\ge 1.
$$
Then
\begin{equation}\label{eq:lem:4-1:3}
\left|\ep[Y_{k+n}Y_k]\right|
\le 4e^{t_0}e^{-t_0n}
\end{equation}
for all $k\ge 1$ and $n\ge 1$.
\end{description}
\end{lemma}

\begin{proof}[Proof of Lemma \ref{lem:4-1}]
It is obvious that
$$
\ep\bigl[Z_{\alpha}^2(s)\bigr]=f(s):=
\sum_{k=1}^{\infty}\frac{1}{k^{1+2s}\log k}.
$$
(i) By the integral test,
$$
f(s)=\int_e^\infty \frac{x^{-1-2s}}{\log x}\,dx+O(1),
\qquad s\to 0.
$$
Making the change of variable $y=\ln x$, we obtain
$$
\int_e^\infty \frac{x^{-1-2s}}{\log x}\,dx
=\int_1^\infty \frac{e^{-2sy}}{y}\,dy
=\int_{2s}^\infty \frac{e^{-y}}{y}\,dy
=E_1(2s)=-\log s+O(1),
$$
where $E_1$ is the exponential integral with
$$
E_1(z)=-\gamma-\log z+O(z),\qquad z\to 0,
$$
and $\gamma$ is the Euler--Mascheroni constant. This proves
\eqref{eq:lem:4-1:1}.

For the increment estimate \eqref{eq:lem:4-1:2}, set
\begin{equation}\label{eq:proof:lem:4-1-1}
u=e^{-s}\in(0,1),\qquad c=1-e^{-t}\in(0,1).
\end{equation}
Then $e^{-s-t}=u(1-c)$. Hence
$$
Z_{\alpha}(e^{-s-t})-Z_{\alpha}(e^{-s})
=
\sum_{k=1}^{\infty}\eta_k
\left(
\frac{k^{-e^{-s-t}-1/2}}{(\log k)^{1/2}}
-
\frac{k^{-e^{-s}-1/2}}{(\log k)^{1/2}}
\right)
=
\sum_{k=1}^{\infty}\eta_k
k^{-u-1/2}(\log k)^{-1/2}(k^{uc}-1).
$$
Therefore, by independence and variance $1$,
$$
\ep\left[\left(Z_{\alpha}(e^{-s-t})-Z_{\alpha}(e^{-s})\right)^2\right]
=
E(u,c):=
\sum_{k=1}^{\infty}
\frac{k^{-1-2u}}{\log k}(k^{uc}-1)^2.
$$
By the integral test, there are absolute constants $a,b>0$ such that
$$
aF(u,c)\le E(u,c)\le bF(u,c),
$$
where
$$
F(u,c)=\int_2^\infty
\frac{x^{-1-2u}}{\ln x}(x^{uc}-1)^2\,dx.
$$
Make the change of variable $y=u\ln x$. Then
$$
F(u,c)=\int_{u\ln 2}^{\infty}
e^{-2y}(e^{cy}-1)^2\frac{dy}{y}.
$$
Since $0<u<1$, we have $0<u\ln 2<\ln 2$. Thus
$$
I(c)-K(c)
=\int_{\ln 2}^{\infty}e^{-2y}(e^{cy}-1)^2\frac{dy}{y}
\le F(u,c)\le I(c),
$$
where
$$
I(c)=\int_0^\infty e^{-2y}(e^{cy}-1)^2\frac{dy}{y},
\qquad
K(c)=\int_0^{\ln 2}e^{-2y}(e^{cy}-1)^2\frac{dy}{y}.
$$
Using the Frullani integral
$$
\int_0^\infty \frac{e^{-ay}-e^{-by}}{y}\,dy=\ln\frac{b}{a},
\qquad a,b>0,
$$
we obtain
$$
I(c)=\ln\frac{1}{1-c}-2\ln\frac{2}{2-c}.
$$
Since $c=1-e^{-t}$, we have $1-c=e^{-t}$ and $2-c=1+e^{-t}$.
Therefore
\begin{equation}\label{eq:proof:lem:4-1-2}
I(c)=t-2\ln 2+2\ln(1+e^{-t}).
\end{equation}
Also,
$$
K(c)\le \int_0^{\ln 2}e^{-2y}(e^{y}-1)^2\frac{dy}{y}
\le \int_0^{\ln 2}y\,dy=\frac{(\ln 2)^2}{2}\le 1.
$$
It follows that
$$
t-2\ln 2\le F(u,c)\le t.
$$
This proves \eqref{eq:lem:4-1:2}.

(ii) For $m\ge 1$, write $\alpha_m=e^{-t_0m}$. Then
$$
Y_m=Z_{\alpha}(\alpha_{m-1})-Z_{\alpha}(\alpha_m)
=
\sum_{j\ge 1}\xi_j
\left(
\frac{j^{-\alpha_{m-1}-1/2}}{(\log j)^{1/2}}
-
\frac{j^{-\alpha_m-1/2}}{(\log j)^{1/2}}
\right)
=
\sum_{j\ge 1}\xi_j j^{-1/2}(\log j)^{-1/2}
\left(j^{-\alpha_{m-1}}-j^{-\alpha_m}\right).
$$
Set $u=e^{-t_0k}$, $v=e^{-t_0(k+n)}$. Then
$\alpha_k=u$, $\alpha_{k-1}=e^{t_0}u$, $\alpha_{k+n}=v$,
and $\alpha_{k+n-1}=e^{t_0}v$. Put $A=e^{t_0}-1$. Since the $\xi_j$ are
independent with mean $0$ and variance $1$,
$$
\ep[Y_{k+n}Y_k]
=
\sum_{j\ge 1}j^{-1}(\log j)^{-1}
\left(j^{-\alpha_{k-1}}-j^{-\alpha_k}\right)
\left(j^{-\alpha_{k+n-1}}-j^{-\alpha_{k+n}}\right)
=
\sum_{j\ge 1}
\frac{j^{-1-u-v}}{\log j}
(1-j^{-Au})(1-j^{-Av}).
$$
Now use the elementary inequalities $1-e^{-x}\le x$ and $1-e^{-x}\le 1$.
With $\varepsilon=u+v$,
$$
0\le \ep[Y_{k+n}Y_k]
\le
\sum_{j\ge 1}
\frac{j^{-1-\varepsilon}}{\log j}
\min\{Au\log j,1\}\,Av\log j
\le
Av\sum_{j\ge 1}j^{-1-\varepsilon}
\le 4Av\int_1^{\infty}x^{-1-\varepsilon}\,dx
=\frac{4Av}{\varepsilon}.
$$
But
$$
\varepsilon=u+v=e^{-t_0k}+e^{-t_0(k+n)}
=e^{-t_0k}(1+e^{-t_0n}),
\qquad
v=e^{-t_0(k+n)}=e^{-t_0k}e^{-t_0n}.
$$
Thus
$$
\big|\ep[Y_{k+n}Y_k]\big|
\le 2A\frac{e^{-t_0n}}{1+e^{-t_0n}}
\le 4Ae^{-t_0n}
=(e^{t_0}-1)e^{-t_0n}
\le 4e^{t_0}e^{-t_0n}.
$$
Therefore \eqref{eq:lem:4-1:3} holds. This proves the lemma.
\end{proof}

The following  is the Fernique inequality \cite{Fernique1975}.
 
\begin{lemma}\label{lem:Fernique}
Let \(\{X(t),\, t\in[0,1]\}\) be a separable, centered Gaussian process with bounded variance
\[
\mathbb{E}[X(t)^2] \leq \Gamma^2
\]
and increment variance satisfying
\[
\mathbb{E}\big[(X(t)-X(s))^2\big] \leq \psi^2(|t-s|),
\]
where \(\psi\) is a nondecreasing continuous function such that
\[
\int_1^\infty \psi(e^{-y^2})\,dy < \infty.
\]
Then, for sufficiently large \(x\) (e.g., \(x \geq 2\)), there exists a constant \(C > 0\) such that
\[
\mathbb{P}\left(\sup_{0 \leq t \leq 1} |X(t)| > x\left(\Gamma + 4\int_1^\infty \psi(e^{-y^2})\,dy\right)\right)
\leq C \int_x^\infty e^{-y^2/2}\,dy.
\]
\end{lemma}

In the proofs of Proposition \ref{prop:2-2} and Theorem \ref{th:3}, we will
use the following elementary estimates.

\begin{lemma}\label{lem:5-1}
Let
$$
f(s)=\int_1^\infty y^\beta e^{-sy}\,dy,
\qquad
g(s)=\int_{1}^{\infty}\sqrt{\log y}\,
\bigl|\alpha y^{\alpha-1}-s y^{\alpha}\bigr|e^{-sy}\,dy,
$$
$s>0$. Then as $s\to 0$,
$$
f(s)\sim
\begin{cases}
\Gamma(\beta+1)s^{-\beta-1}, & \beta>-1,\\[2mm]
\log\dfrac1s, & \beta=-1,\\[2mm]
-\dfrac{1}{\beta+1}, & \beta<-1,
\end{cases}
$$
and
$$
g(s)\sim
\begin{cases}
s^{-\alpha}\sqrt{\log(1/s)}
\displaystyle\int_{0}^\infty t^{\alpha-1}|\alpha-t|e^{-t}\,dt,
& \alpha\ge 0,\\[4mm]
|\alpha|\displaystyle\int_1^\infty v^{\alpha-1}\sqrt{\log v}\,dv,
& \alpha<0.
\end{cases}
$$
\end{lemma}

\begin{proof}
We first consider $f(s)$. Make the substitution $t=sy$. Then
$$
f(s)=\int_1^\infty y^\beta e^{-sy}\,dy
=\int_s^\infty \left(\frac{t}{s}\right)^\beta e^{-t}\frac{dt}{s}
=s^{-\beta-1}\int_s^\infty t^\beta e^{-t}\,dt
=s^{-\beta-1}\Gamma(\beta+1,s).
$$
We now use the standard asymptotic behaviour of the upper incomplete Gamma
function as $s\to0$.
\begin{description}
\item[\rm (i)] If $\beta>-1$, then
$$
\Gamma(\beta+1,s)
=\Gamma(\beta+1)-\frac{s^{\beta+1}}{\beta+1}+O(s^{\beta+2}).
$$
Thus $f(s)\sim \Gamma(\beta+1)s^{-\beta-1}$.
\item[\rm (ii)] If $\beta=-1$, then
$$
f(s)=\Gamma(0,s)=E_1(s)=-\log s-\gamma+O(s),\qquad s\to 0.
$$
\item[\rm (iii)] If $\beta<-1$, then by the dominated convergence theorem,
$$
f(s)=\int_1^\infty y^{\beta}e^{-sy}\,dy
\to \int_1^\infty y^{\beta}\,dy=-\frac1{\beta+1}.
$$
\end{description}

For considering $g(s)$, write $L=\log(1/s)$. Put $t=sy$. Then
$$
g(s)=s^{-\alpha}\int_{s}^{\infty}
t^{\alpha-1}|\alpha-t|\sqrt{L+\log t}\,e^{-t}\,dt
=
s^{-\alpha}\sqrt{L}\int_{s}^\infty
t^{\alpha-1}|\alpha-t|
\sqrt{1+\frac{\log t}{L}}\,e^{-t}\,dt.
$$
When $\alpha\ge0$,
$$
\int_0^{\infty}
t^{\alpha-1}|\alpha-t|\sqrt{1+|\log t|}\,e^{-t}\,dt<\infty.
$$
By the dominated convergence theorem,
$$
\int_{s}^\infty t^{\alpha-1}|\alpha-t|
\sqrt{1+\frac{\log t}{L}}\,e^{-t}\,dt
\to \int_{0}^\infty t^{\alpha-1}|\alpha-t|e^{-t}\,dt<\infty.
$$
Now consider the case $\alpha<0$. Then $|\alpha-sy|=sy-\alpha$. The
integrand is dominated by
$$
y^{\alpha-1}(|\alpha|+y)\sqrt{\log y},
$$
which is integrable for $\alpha<0$ because $\alpha-1<-1$. By the dominated
convergence theorem,
$$
g(s)=
\int_1^\infty y^{\alpha-1}|\alpha-sy|\sqrt{\log y}\,e^{-sy}\,dy
\to |\alpha|\int_1^\infty y^{\alpha-1}\sqrt{\log y}\,dy.
$$
\end{proof}

\begin{proof}[Proof of Proposition \ref{prop:2-1}]
Let $\alpha=-1/2$. By \eqref{eq:lem:4-1:2}, we can choose $t_0>0$ large
enough such that
$$
\ep\left[\big(Z_{\alpha}(e^{-s-t})-Z_{\alpha}(e^{-s})\big)^2\right]
\ge tC_1/2,\qquad t\ge t_0,\ s>0.
$$
Let $Y_n=Z_{\alpha}(e^{-t_0(n+1)})-Z_{\alpha}(e^{-t_0n})$. Then
$\{Y_n;n\ge 1\}$ is a centered Gaussian sequence with
\begin{equation}\label{eq:proof:Gauss:ap1}
\ep\left(\sum_{k=m+1}^{n+m}Y_k\right)^2\ge nt_0C_1/2
\quad\text{for all }m,n,
\end{equation}
\begin{equation}\label{eq:proof:Gauss:ap2}
\ep Y_n^2\le C_2(t_0+1)
\end{equation}
by \eqref{eq:lem:4-1:2}, and
\begin{equation}\label{eq:proof:Gauss:ap3}
\gamma(n)=\sup_k|\ep Y_kY_{k+n}|\le Ce^{-t_0n}
\end{equation}
by \eqref{eq:lem:4-1:3}. Applying an almost sure invariance principle of
Shao \cite{Shao1985} (cf. Corollary 14.2.1 of Lin and Lu \cite{LinLu1997}),
for the Gaussian sequence $\{Y_n;n\ge 1\}$, under
\eqref{eq:proof:Gauss:ap1}--\eqref{eq:proof:Gauss:ap3} we can find a
standard Brownian motion $B(t)$ such that
$$
\sum_{k=1}^n Y_k-B(s_n^2)=O\big(\log^{1/2}n\big)\quad a.s.,
$$
where
$$
s_n^2=\ep\left(\sum_{k=1}^n Y_k\right)^2
=\ep Z^2(e^{-t_0(n+1)})+O(1)
=t_0(n+1)+O(1)
$$
by \eqref{eq:lem:4-1:1}. Notice
$$
\sup_{t_0n\le T\le t_0(n+1)}|s_n^2-T|=O(1).
$$
By the sample path property of a Brownian motion (cf. Theorem 1.2.1 of
Cs\"org\H{o} and R\'ev\'esz \cite{CR81}), we have
$$
\sup_{t_0n\le T\le t_0(n+1)}\Big|B(s_n^2)-B(T)\Big|
=O(\sqrt{\log n})\quad a.s.
$$
Notice that $Z_{\alpha}(e^{-t})$ is a Gaussian process which satisfies
\eqref{eq:lem:4-1:2}. By the Fernique inequality (Lemma \ref{lem:Fernique}), we have 
\begin{align*}
&\pr\Big(\sup_{t_0n\le T\le t_0(n+1)}
\Big|Z_{\alpha}(e^{-T})-Z_{\alpha}(e^{-t_0(n+1)})\Big|
\ge 2\sqrt{\log n}\sqrt{C_2t_0}\big(1+\int_1^{\infty}e^{-y^2/2}dy\big)\Big)\\
&\le C\int_{2\sqrt{\log n}}^{\infty}e^{-y^2/2}dy\le C n^{-2}. 
\end{align*}
Thus,
\begin{equation}
\label{eq:proof:prop:2-1-1}
\sup_{t_0n\le T\le t_0(n+1)}
\Big|Z_{\alpha}(e^{-t_0(n+1)})-Z_{\alpha}(e^{-T})\Big|
=O\Big(\sqrt{\log n}\Big)\quad a.s.
\end{equation}
cf. also  
Corollary 3.2.1 of Lin et al. \cite{LinLuZhang2001}. 
It follows that
\begin{equation}
\label{eq:proof:prop:2-1-2}
Z_{\alpha}(e^{-T})-B(T)=O\Big(\sqrt{\log T}\Big)\quad a.s.
\text{ as }T\to\infty.
\end{equation}
This proves \eqref{eq:prop:2-1.2}.
\end{proof}

\begin{proof}[Proof of Proposition \ref{prop:2-2}]
Write $f_s(x)=(\log x)^{\alpha}x^{-s-1/2}$. Then
$$
f_s^{\prime}(x)=
\frac{(\log x)^{\alpha}
(-s-1/2+\alpha(\log x)^{-1})}{x^{s+3/2}},
\qquad
|f_s^{\prime\prime}(x)|\le C(\log x)^{\alpha}x^{-s-5/2}.
$$
Without loss of generality, we assume $\eta_k=W(k)-W(k-1)$, where $W(t)$ is
a standard Brownian motion. Then
\begin{align*}
Z_{\alpha}(s)
&=\sum_{k=1}^{\infty}(W(k)-W(k-1))f_s(k)
=\sum_{k=1}^{\infty}W(k)\Big(f_s(k)-f_s(k-1)\Big)\\
&=\int_2^{\infty} W(x)f_s^{\prime}(x)\,dx
+\sum_{k=1}^{\infty}\int_k^{k+1}
\Big[W(k)-W(x)\Big]\Big(f_s(k)-f_s(k-1)\Big)\,dx\\
&\quad+\sum_{k=1}^{\infty}\int_k^{k+1}
W(x)\Big(f_s(k)-f_s(k-1)-f_s^{\prime}(x)\Big)\,dx\\
&=:\int_1^{\infty} W(x)f_s^{\prime}(x)\,dx+I_1(s)+I_2(s).
\end{align*}
Since $\sup_{k\le x\le k+1}|W(x)-W(k)|=O(\sqrt{\log k})$ a.s. and
$|W(x)|=O(\sqrt{x\log\log x})$ a.s., we have
\begin{align*}
|I_1(s)|
&\le C\sum_{k=1}^{\infty}\sqrt{\log k}(\log k)^{\alpha}k^{-s-3/2}
\le C\sum_{k=2}^{\infty}\sqrt{\log k}(\log k)^{\alpha}k^{-3/2}
\le C\quad a.s.,\\
|I_2(s)|
&\le C\sum_{k=1}^{\infty}\sqrt{k\log\log k}(\log k)^{\alpha}k^{-s-5/2}
\le C\sum_{k=1}^{\infty}\sqrt{\log\log k}(\log k)^{\alpha}k^{-3/2}
\le C\quad a.s.
\end{align*}
Also,
\begin{align*}
&\int_e^{\infty}|W(x)|\cdot
|\alpha(\log x)^{\alpha-1}-s(\log x)^{\alpha}|x^{-s-3/2}\,dx\\
&\le \int_e^{\infty}\sqrt{x\log\log x}
|\alpha(\log x)^{\alpha-1}-s(\log x)^{\alpha}|x^{-s-3/2}\,dx\\
&=\int_1^{\infty}\sqrt{\log y}
|\alpha y^{\alpha-1}-sy^{\alpha}|e^{-sy}\,dy
=\begin{cases}
O\Big(s^{-\alpha}\sqrt{\log(1/s)}\Big),&\alpha\ge0,\\
O(1),&\alpha<0,
\end{cases}
\end{align*}
by Lemma \ref{lem:5-1}. Let $U(x)=W(e^x)e^{-x/2}$,
$V(x)=\int_1^x U(t)\,dt$. It follows that
$$
Z_{\alpha}(s)
=\frac{1}{2}\int_e^{\infty}W(x)(\log x)^{\alpha}x^{-s-3/2}\,dx
+O\Big(s^{-\alpha\vee 0}\sqrt{\log(1/s)}\Big)
$$
and hence
$$
Z_{\alpha}(s)
=\frac{1}{2}\int_1^{\infty}U(y)y^{\alpha}e^{-sy}\,dy
+O\Big(s^{-\alpha\vee 0}\sqrt{\log(1/s)}\Big)\quad a.s.
$$
On the other hand, notice that $U(t)$ is a stationary Gaussian process with
covariance function $\Cov\{U(s+t),U(s)\}=e^{-|t|/2}$. It follows that
$$
\ep(V(s+t)-V(s))^2
=\ep\left(\int_0^t U(v)\,dv\right)^2
=\int_0^t\int_0^t e^{-|u-v|/2}\,du\,dv
=4t-8(1-e^{-t/2})\le 4t.
$$
If we let $Y_n=V(t_0n)-V(t_0(n-1))$, it can be shown that $Y_n$ satisfies
\eqref{eq:proof:Gauss:ap1}--\eqref{eq:proof:Gauss:ap3}. Especially,
$$
\gamma(n)=|\ep Y_{k+n}Y_k|
=\int_{t_0n}^{t_0(n+1)}\int_{0}^{t_0}e^{-|u-v|/2}\,dv\,du
\le 4e^{-t_0n/2}.
$$
By the same argument for showing \eqref{eq:proof:prop:2-1-2}, there exists a
Brownian motion $B(t)$ such that
$$
V(t)-2B(t)=O(\sqrt{\log t})\quad a.s.\text{ as }t\to\infty.
$$
cf. also Zhang \cite[Section 2]{Zhang2026}. It follows that
\begin{align*}
\int_1^{\infty}U(y)y^{\alpha}e^{-sy}\,dy
&=\int_1^{\infty}y^{\alpha}e^{-sy}\,dV(y)
=\int_1^{\infty}V(y)\Big(\alpha y^{\alpha-1}-sy^{\alpha}\Big)e^{-sy}\,dy\\
&=2\int_1^{\infty}B(y)\Big(\alpha y^{\alpha-1}-sy^{\alpha}\Big)e^{-sy}\,dy\\
&\quad+O(1)\int_1^{\infty}\sqrt{\log y}
\Big|\alpha y^{\alpha-1}-sy^{\alpha}\Big|e^{-sy}\,dy\\
&=2\int_1^{\infty}y^{\alpha}e^{-sy}\,dB(y)
+O\Big(s^{-\alpha\vee0}\sqrt{\log(1/s)}\Big)\quad a.s.
\end{align*}
Notice that
$$
\int_0^1 y^{\alpha}e^{-sy}\,dB(y)
=B(1)e^{-s}-\int_0^1B(y)(\alpha y^{\alpha-1}-sy^{\alpha})e^{-sy}\,dy
=O(1)\quad a.s.
$$
We have
$$
Z_{\alpha}(s)=\int_0^{\infty}y^{\alpha}e^{-sy}\,dB(y)
+O\Big(s^{-\alpha\vee0}\sqrt{\log(1/s)}\Big)\quad a.s.
$$
The proof of the proposition is completed.
\end{proof}

\section{Random Dirichlet series: the critical case}\label{sec:critical}
\setcounter{equation}{0}

In this section, we give the proof of Theorem \ref{th:2}. We first show
\eqref{eq:th:2.3} under
$\ep[\xi^2(\log|\xi|)^{1+\epsilon}]<\infty$. Set
$C_k(s)=k^{-s-1/2}(\log k)^{-1/2}$, $S_n=\sum_{k=1}^n\xi_k$,
$H(x)=x^2(\log x)^{1+\epsilon}$. Then
$\operatorname{inv}H(x)\sim \sqrt{x}(\frac12\log x)^{-1/2-\epsilon/2}$ as
$x\to\infty$. By the strong approximation of independent and identically
distributed random variables (cf. Chapter 2 of Cs\"org\H{o} and R\'ev\'esz
\cite{CR81}), there exists a standard Brownian motion $W(t)$ such that
\begin{equation}\label{eq:proof:th1.0}
S_n-\sigma W(n)
=o\big(n^{1/2}(\log n)^{-1/2-\epsilon/2}\big)\quad a.s.
\end{equation}
Write $\eta_k=W(k)-W(k-1)$. Let
$f(x)=x^{1/2}(\log x)^{-1/2-\epsilon/2}$. Then
$|f(k)-f(k-1)|\le Ck^{-1/2}(\log k)^{-1/2-\epsilon/2}$. By Abel summation,
\begin{align}\label{eq:proof:th2.1}
&\left|\sum_{k=1}^{\infty}C_k(s)(\xi_k-\sigma\eta_k)\right|
=\lim_{N\to\infty}\left|
\sum_{k=1}^N C_k(s)(\xi_k-\sigma\eta_k)\right|\nonumber\\
&=\lim_{N\to\infty}\left|
C_N(s)[S_N-\sigma W(N)]
+\sum_{k=1}^{N-1}[C_k(s)-C_{k+1}(s)][S_k-\sigma W(k)]
\right|\nonumber\\
&=\left|\sum_{k=1}^{\infty}[C_k(s)-C_{k+1}(s)][S_k-\sigma W(k)]\right|
\nonumber\\
&\le C\sum_{k=1}^{\infty}[C_k(s)-C_{k+1}(s)]f(k)
=C\lim_{N\to\infty}\sum_{k=1}^{N-1}[C_k(s)-C_{k+1}(s)]f(k)
\nonumber\\
&=C\lim_{N\to\infty}\left[
-\frac{f(N)}{N^{s+1/2}(\log N)^{1/2}}
+\sum_{k=1}^{N}C_k(s)[f(k)-f(k-1)]
\right]\nonumber\\
&=C\sum_{k=1}^{\infty}C_k(s)[f(k)-f(k-1)]\nonumber\\
&\le C\sum_{k=1}^{\infty}C_k(s)k^{-1/2}(\log k)^{-1/2-\epsilon/2}
\le C\sum_{k=1}^{\infty}
\frac{k^{-1/2}(\log k)^{-1/2-\epsilon/2}}{(k\log k)^{1/2}}
=O(1)\quad a.s.
\end{align}
This proves \eqref{eq:th:2.3} by Proposition \ref{prop:2-1}.

\bigskip
For showing \eqref{eq:th:2.1} and \eqref{eq:th:2.2}, we need more lemmas.

\begin{lemma}\label{lem:approx}
(Sakhanenko \cite{Sakhanenko1980,Sakhanenko1984,Sakhanenko1985})
For any sequence of independent random variables $\{X_n;n\ge1\}$ with mean
zero and finite variance, there exists a sequence of independent normal
variables $\{Y_n;n\ge1\}$ with $\ep Y_n=0$ and
$\ep Y_n^2=\ep X_n^2$, such that for all $Q>2$ and $y>0$,
$$
\pr\left(\max_{m\le n}\left|\sum_{k=1}^m(X_k-Y_k)\right|\ge y\right)
\le (AQ)^Q y^{-Q}\sum_{k=1}^n\ep|X_k|^Q.
$$
Here $A$ is a universal constant.
\end{lemma}

\begin{lemma}\label{lem:2}
Let $\{u_n;n\ge1\}$ be a non-increasing sequence of positive numbers and
$\{x_n,n\ge1\}$ be a sequence of real numbers. Then for each $n\ge1$,
$$
\left|\sum_{i=1}^n u_ix_i\right|
\le u_1\max_{m\le n}\left|\sum_{i=1}^m x_i\right|.
$$
\end{lemma}

\begin{proof}
The proof follows from Abel summation and is omitted.
\end{proof}

\begin{lemma}\label{lem:3}
Let $\{\xi_p;p\ge1\}$ be a sequence of independent random variables with
mean zero and $\ep\xi_p^2\to0$ as $p\to\infty$. Let $Y(s)$ be defined as in
\eqref{eq:1} with $\alpha=-1/2$. Then
\begin{equation}\label{eq:lem:3.1}
\ep\left[\sup_{s\le t\le1}|Y(t)|^2\right]
=o(\log(1/s))\quad\text{as }s\to0.
\end{equation}
Furthermore, if
\begin{equation}\label{eq:lem:3.2}
\pr(|\xi_p|\ge x)\le C\pr(|\eta|\ge x),\qquad x\ge0,
\quad \ep\eta^2<\infty,
\end{equation}
then
\begin{equation}\label{eq:lem:3.3}
\lim_{s\to0}
\frac{|Y(s)|}{\sqrt{\log(1/s)\log\log\log(1/s)}}=0.
\end{equation}
\end{lemma}

\begin{proof}
For $0<s\le t\le1$, by Lemma \ref{lem:2},
$$
|Y(t)|
=\left|\sum_{k=2}^{\infty}k^{s-t}
\frac{\xi_k}{k^{s+1/2}(\log k)^{1/2}}\right|
\le \max_{m\ge1}
\left|\sum_{k=2}^m
\frac{\xi_k}{k^{s+1/2}(\log k)^{1/2}}\right|.
$$
Thus,
\begin{equation}\label{eq:proof:lem3.1}
\sup_{s\le t\le1}|Y(t)|
=\max_{m\ge1}
\left|\sum_{k=2}^m
\frac{\xi_k}{k^{s+1/2}(\log k)^{1/2}}\right|.
\end{equation}
It follows that
$$
\ep\sup_{s\le t\le1}|Y(t)|^2
\le 2\sum_{k=2}^{\infty}
\frac{\ep\xi_k^2}{k^{2s+1}(\log k)}
=o(\log(1/s)).
$$
\eqref{eq:lem:3.1} is proved. For \eqref{eq:lem:3.3}, write
$h(s)=\log(1/s)\log\log\log(1/s)$,
$$
\overline{\xi}_k=\xi_k I\{|\xi_k|<k^{1/2}\}
-\ep[\xi_k I\{|\xi_k|<k^{1/2}\}],
\qquad
\overline{Y}(s)=\sum_{k=2}^{\infty}
\frac{\overline{\xi}_k}{k^{s+1/2}(\log k)^{1/2}}.
$$
Notice that
$$
\sum_{k=1}^{\infty}\pr(|\xi_k|\ge k^{1/2})
\le C\sum_{k=1}^{\infty}\pr(|\eta|\ge k^{1/2})<\infty
$$
by \eqref{eq:lem:3.2}. Hence
$\pr(|\xi_k|/k^{1/2}\ge1\ \text{i.o.})=0$. It follows that
$$
\sum_{k=2}^{\infty}
\frac{|\xi_k|I\{|\xi_k|\ge k^{1/2}\}}
{k^{s+1/2}(\log k)^{1/2}}
\le \sum_{k=2}^{\infty}
\frac{|\xi_k|I\{|\xi_k|\ge k^{1/2}\}}{k^{1/2}}
=O(1)\quad a.s.
$$
Also,
\begin{align*}
&\sum_{k=2}^{\infty}
\frac{\ep[|\xi_k|I\{|\xi_k|\ge k^{1/2}\}]}
{k^{s+1/2}(\log k)^{1/2}}\\
&\le C\sum_{k=2}^{\infty}
\frac{\ep[|\eta|I\{|\eta|\ge k^{1/2}\}]}
{k^{s+1/2}(\log k)^{1/2}}
\le C\sum_{k=2}^{\infty}
\frac{\ep[|\eta|I\{|\eta|\ge k^{1/2}\}]}{k^{1/2}}\\
&\le C\ep\left[|\eta|\sum_{k^{1/2}\le|\eta|}k^{-1/2}\right]
\le C\ep\eta^2.
\end{align*}
It remains to show that
\begin{equation}\label{eq:proof:lem3.2}
\lim_{s\to0}
\frac{|\overline{Y}(s)|}{\sqrt{h(s)}}=0.
\end{equation}
It is easily seen that
$$
\Var(\xi_k-\overline{\xi}_k)
\le \ep[|\xi_k|^2I\{|\xi_k|\ge k^{1/2}\}]
\le C\ep[|\eta|^2I\{|\eta|\ge k^{1/2}\}]\to0.
$$
Thus
$$
\Var\{Y(s)-\overline{Y}(s)\}=o(\sqrt{\log(1/s)}),
\qquad
\Var\{\overline{Y}(s)\}=o(\sqrt{\log(1/s)}).
$$
Notice that
$|\overline{\xi}_k|/(k^{s+1/2}(\log k)^{1/2})\le2$. By the Bernstein
inequality for the maximum of sums of independent bounded random variables,
\begin{align*}
&\pr\left(\sup_{s\le t\le1}|\overline{Y}(t)|
\ge \epsilon\sqrt{h(s)}\right)
=
\pr\left(\sup_{m\ge1}\left|
\sum_{k\le m}
\frac{\overline{\xi}_k}{k^{s+1/2}(\log k)^{1/2}}
\right|\ge \epsilon\sqrt{h(s)}\right)\\
&\le 2\exp\left[
-\frac{\epsilon^2h(s)}
{2(\epsilon\sqrt{h(s)}\times2+\Var\{\overline{Y}(s)\})}
\right]
=2\exp\left\{-\frac{\epsilon^2h(s)}{o(\log(1/s))}\right\}
\le \exp\{-3\log\log\log(1/s)\}
\end{align*}
for $s$ small enough. Choose $s_n=\exp\{-e^n\}$. Then
$$
\sum_{n=1}^{\infty}
\pr\left(\sup_{s_n\le t\le1}|\overline{Y}(t)|
\ge\epsilon\sqrt{h(s_n)}\right)
\le C\sum_{n=1}^{\infty}n^{-3}<\infty,
$$
which implies that
$$
\lim_{n\to\infty}
\frac{\sup_{s_n\le t\le1}|\overline{Y}(t)|}{\sqrt{h(s_n)}}=0\quad a.s.
$$
For $s_n\le s\le s_{n-1}$,
$$
\frac{|\overline{Y}(s)|}{\sqrt{h(s)}}
\le
\frac{\sqrt{h(s_n)}}{\sqrt{h(s_{n-1})}}
\frac{\sup_{s_n\le t\le1}|\overline{Y}(t)|}{\sqrt{h(s_n)}}
\to0\quad a.s.
$$
This proves \eqref{eq:proof:lem3.2}. The proof of the lemma is completed.
\end{proof}

\begin{proof}[Continuation of the proof of Theorem \ref{th:2}]
Now we show \eqref{eq:th:2.1} and \eqref{eq:th:2.2}. Let
$c_k=\log k$,
$$
\overline{\xi}_k=\xi_kI\{|\xi_k|\le c_k\}
-\ep[\xi_kI\{|\xi_k|\le c_k\}],
\qquad
\overline{\sigma}_k=\Var(\overline{\xi}_k),
\qquad
h(s)=\log(1/s)\log\log\log(1/s).
$$
By Lemma \ref{lem:approx}, there is a sequence of independent $N(0,1)$
random variables $\{\eta_k;k=1,2,\ldots\}$ such that
\begin{align}\label{eq:proof:th2.2}
&\pr\left(\max_{m\le n}
\big|\sum_{k=1}^m(\overline{\xi}_k-\overline{\sigma}_k\eta_k)\big|
\ge y\right)
\le (AQ)^Q y^{-Q}\sum_{k=1}^n\ep|\overline{\xi}_k|^Q\nonumber\\
&\le (AQ)^Q y^{-Q} n(\log n)^Q
\le c n(\log n)^Q y^{-Q}.
\end{align}
Letting $y=\sqrt{x}n^{1/4}$ and $Q=16$ yields
$$
\pr\left(\sup_{n\ge1}\max_{m\le n}
\big|\sum_{k=1}^m(\overline{\xi}_k-\overline{\sigma}_k\eta_k)\big|/n^{1/4}
\ge\sqrt{x}\right)
\le C\sum_{n=1}^{\infty}n(\log n)^{16}n^{-4}x^{-8}
\le Cx^{-8}.
$$
It follows that
\begin{equation}\label{eq:proof:th2.3}
\ep\left[
\sup_{n\ge1}\max_{m\le n}
\big|\sum_{k=1}^m(\overline{\xi}_k-\overline{\sigma}_k\eta_k)\big|/n^{1/4}
\right]^2
\le 1+\int_1^{\infty}Cx^{-8}\,dx\le C+1.
\end{equation}
Hence
$$
\max_{m\le n}
\big|\sum_{k=1}^m(\overline{\xi}_k-\overline{\sigma}_k\eta_k)\big|
=O(n^{1/4})\quad a.s.\text{ and in }L_2.
$$
Write $T_k=\sum_{i=1}^k[\overline{\xi}_i-\overline{\sigma}_i\eta_i]$.
Similar to \eqref{eq:proof:th2.1},
\begin{align}\label{eq:proof:th2.5}
&\left|\sum_{k=1}^{\infty}C_k(s)(\overline{\xi}_k-\overline{\sigma}_k\eta_k)\right|
=\left|\sum_{k=1}^{\infty}[C_k(s)-C_{k+1}(s)]T_k\right|\nonumber\\
&\le C\sum_{k=1}^{\infty}[C_k(s)-C_{k+1}(s)]k^{1/4}
=C\sum_{k=1}^{\infty}C_k(s)[k^{1/4}-(k-1)^{1/4}]\nonumber\\
&\le C\sum_{k=1}^{\infty}\frac{k^{-3/4}}{(k\log k)^{1/2}}
=O(1)\quad a.s.
\end{align}
Notice that $\Var(\xi_k-\overline{\xi}_k)\to0$,
$|\xi_k-\overline{\xi}_k|\le|\xi_k|+C$,
$|\sigma-\overline{\sigma}_k|\le2\sigma$ and
$\overline{\sigma}_k\to\sigma$. By \eqref{eq:lem:3.3} of Lemma \ref{lem:3},
$$
\lim_{s\to0}\frac{1}{\sqrt{h(s)}}
\sum_{k=2}^{\infty}
\frac{\xi_k-\overline{\xi}_k}{k^{s+1/2}(\log k)^{1/2}}=0\quad a.s.
$$
and
$$
\lim_{s\to0}\frac{1}{\sqrt{h(s)}}
\sum_{k=2}^{\infty}
\frac{(\sigma-\overline{\sigma}_k)\eta_k}
{k^{s+1/2}(\log k)^{1/2}}=0\quad a.s.
$$
It follows that
$$
\lim_{s\to0}
\frac{Y(s)-\sigma Z_{\alpha}(s)}{\sqrt{h(s)}}=0\quad a.s.,
$$
where
$$
Z_{\alpha}(s)=\sum_{k=2}^{\infty}
\frac{\eta_k}{k^{s+1/2}(\log k)^{1/2}}.
$$
This proves \eqref{eq:th:2.1} by Proposition \ref{prop:2-1}.

At last, we consider \eqref{eq:th:2.2}. By Abel summation and
\eqref{eq:proof:lem3.1}, \eqref{eq:proof:th2.3},
\begin{align*}
&\max_{s\le t\le1}\left|
\sum_{k=1}^{\infty}
\frac{\overline{\xi}_k-\overline{\sigma}_k\eta_k}
{k^{t+1/2}(\log k)^{1/2}}
\right|\\
&\le \max_{N\ge1}\left|
\sum_{k=1}^{N}
\frac{\overline{\xi}_k-\overline{\sigma}_k\eta_k}
{k^{t+1/2}(\log k)^{1/2}}
\right|\\
&=\max_{N\ge1}\left|
C_N(s)T_N+\sum_{k=1}^{N-1}[C_k(s)-C_{k+1}(s)]T_k
\right|
\le \max_{N\ge1}\frac{|T_N|}{N^{1/2}}
+\sum_{k=1}^{\infty}[C_k(s)-C_{k+1}(s)]|T_k|\\
&\le C+C\sum_{k=1}^{\infty}[C_k(s)-C_{k+1}(s)]k^{1/4}
\quad\text{in }L_2\\
&\le C+C\sum_{k=1}^{\infty}C_k(s)[k^{1/4}-(k-1)^{1/4}]
\le C+C\sum_{k=1}^{\infty}\frac{k^{-3/4}}{k^{1/2}}=O(1).
\end{align*}
By \eqref{eq:lem:3.1} of Lemma \ref{lem:3},
$$
\sup_{s\le t\le1}\left|
\sum_{k=2}^{\infty}
\frac{\xi_k-\overline{\xi}_k}{k^{t+1/2}(\log k)^{1/2}}
\right|
=o\big(\sqrt{\log(1/s)}\big)\quad\text{in }L_2
$$
and
$$
\sup_{s\le t\le1}\left|
\sum_{k=2}^{\infty}
\frac{(\sigma-\overline{\sigma}_k)\eta_k}
{k^{t+1/2}(\log k)^{1/2}}
\right|
=o\big(\sqrt{\log(1/s)}\big)\quad\text{in }L_2.
$$
It follows that
\begin{equation}\label{eq:proof:th2.4}
\sup_{s\le t\le1}|Y(t)-\sigma Z_{\alpha}(t)|
=o\big(\sqrt{\log(1/s)}\big)\quad\text{in }L_2.
\end{equation}
On the other hand, by Proposition \ref{prop:2-1},
\begin{equation}\label{eq:proof:th1.5}
\sup_{s\le t\le1}|Z_{\alpha}(t)-B(\log(1/t))|
=o\big(\sqrt{\log(1/s)}\big)\quad a.s.
\end{equation}
Notice that both $Z_{\alpha}(t)$ and $B(\log(1/t))$ are Gaussian processes.
We have
$$
\ep\left[\sup_{s\le t\le1}|Z_{\alpha}(t)|^4\right]
\le \ep\left[
\max_{m\ge1}\Big|
\sum_{k=2}^{\infty}
\frac{\eta_k}{k^{s+1/2}(\log k)^{1/2}}
\Big|^4
\right]
\le 2\ep[|Z_{\alpha}(s)|^4]
\le C(\log(1/s))^2,
$$
and similarly,
$$
\ep\left[\sup_{s\le t\le1}|B(\log(1/t))|^4\right]
\le 2\ep[|B(\log(1/s))|^4]
\le C(\log(1/s))^2.
$$
It follows that
$$
\sup_{s\le t\le1}|Z_{\alpha}(t)-B(\log(1/t))|
=O\big(\sqrt{\log(1/s)}\big)\quad\text{in }L_4,
$$
which, together with \eqref{eq:proof:th1.5} and the dominated convergence
theorem, implies that
\begin{equation}\label{eq:proof:th2.6}
\sup_{s\le t\le1}|Z_{\alpha}(t)-B(\log(1/t))|
=o\big(\sqrt{\log(1/s)}\big)\quad\text{in }L_2.
\end{equation}
Now \eqref{eq:th:2.2} follows from \eqref{eq:proof:th2.4} and
\eqref{eq:proof:th2.6}. The proof is completed.
\end{proof}

\section{Random prime Dirichlet series}\label{sec:prime}
\setcounter{equation}{0}

In this section, we consider the random prime Dirichlet series. Let $p_k$ be
the $k$-th prime. Then
$$
X(s)=\sum_{p\in \mathscr{P}}\frac{\xi_p}{p^{s+1/2}}=\sum_{k=1}^{\infty}\frac{\xi_{p_k}}{p_k^{s+1/2}}.
$$
We need to show that $X(s)$ has the same asymptotic behavior as the random
Dirichlet series in the critical case:
$$
Y(s)=\sum_{k=1}^{\infty}
\frac{\xi_k}{k^{s+1/2}(\log k)^{1/2}}.
$$
As mentioned by Iksanov and Kostohryz \cite{IK2025}, we have no way to
effectively estimate the difference $X(s)-Y(s)$. In fact,
\begin{align*}
&\Var\big(X(s)-Y(s)\big)
\ge \sum_{p\not\in\mathscr{P}}
\frac{\sigma^2}{p^{2s+1}\log p}
=\Var(Y(s))-\sigma^2\sum_{p\in\mathscr{P}}
\frac{1}{p^{2s+1}\log p}\\
=&\Var(Y(s))-\sigma^2\sum_{k=1}^{\infty}
\frac{1}{p_k^{2s+1}\log p_k}
\ge \Var(Y(s))-\sigma^2C\sum_{k=1}^{\infty}
\frac{1}{k(\log k)^2}
\ge \Var(Y(s))-C\sigma^2.
\end{align*}
However, if we define
\begin{align}\label{eq:primecas1}
Y_1(s)=\sum_{k=1}^{\infty}\frac{\xi_k}{p_k^{s+1/2}},
\qquad
Y_2(s)=\sum_{k=1}^{\infty}
\frac{\xi_k}{(k\log k)^{s+1/2}},
\end{align}
then
\begin{equation}\label{eq:id}
\bigg\{Y_1(s);s>0\bigg\}\overset{d}=\Big\{X(s);s>0\Big\}.
\end{equation}
Then the following theorem gives effective estimates of $Y_1(s)-Y(s)$ and
$Y_2(s)-Y(s)$, and Theorem \ref{th:1} follows from it with Theorem
\ref{th:2} and the fact \eqref{eq:id} directly.

\begin{theorem}\label{th:4}
We have
\begin{equation}
\label{eq:th:4.1}
Y_i(s)-Y(s)
=o\Big(\sqrt{\log(1/s)\log\log\log(1/s)}\Big)\quad a.s.
\text{ as }s\to 0,
\end{equation}
and
\begin{equation}
\label{eq:th:4.2}
\sup_{s\le t\le1}\Big|Y_i(t)-Y(t)\Big|
=o\Big(\sqrt{\log(1/s)}\Big)\quad \text{in }L_2
\text{ as }s\to 0,
\end{equation}
$i=1,2$. Furthermore, if \eqref{eq:th:1.3} holds, then
\begin{equation}
\label{eq:th:4.3}
Y_i(s)-Y(s)=O\Big(\sqrt{\log\log(1/s)}\Big)\quad a.s.
\text{ as }s\to 0,\quad i=1,2.
\end{equation}
Thus, there exists a standard Brownian motion $B(t)$ such that
\eqref{eq:th:2.1}--\eqref{eq:th:2.3} holds simultaneously for $Y(s)$,
$Y_1(s)$ and $Y_2(s)$.
\end{theorem}

This theorem tells us why $X(s)$ and $Y(s)$ have the same FCLT and LIL as
Iksanov and Kostohryz \cite{IK2025} conjectured.

\bigskip
To prove Theorem \ref{th:4}, we need two more lemmas on the variances and
covariances of the increments of $Y_1(s)$ and $Y_2(s)$, respectively.

\begin{lemma}\label{lem:2-1}
Let $Y_1(s)$ be defined as in \eqref{eq:primecas1}. Then
\begin{equation}\label{eq:lem:2-1:2}
\ep[Y_1^2(s)]=\ep[X^2(s)]
=\sigma^2\log(1/s)+O(1)\quad\text{as }s\to0,
\end{equation}
and there exists a constant $C_2>0$ such that
\begin{equation}\label{eq:lem:2-1:3}
\ep\left[\left(Y_1(e^{-s-t})-Y_1(e^{-s})\right)^2\right]
=
\ep\left[\left(X(e^{-s-t})-X(e^{-s})\right)^2\right]
\le C_2t\sigma^2
\end{equation}
for all $t>0$ and $s>0$.
\end{lemma}

\begin{proof}
Without loss of generality, we assume $\sigma=1$. Let
$$
f(x)=\pi(x)-\int_2^{x}\frac{dt}{\ln t},
$$
where $\pi(x)=\#\{p\le x:p\text{ is prime}\}$. By the prime number theorem,
$f(x)=O(x/(\ln x)^2)$. It follows that
$$
\int_{2}^{\infty}x^{-1-2s}\,df(x)
=x^{-1-2s}f(x)\big|_{x=2}^{\infty}
+(1+2s)\int_2^{\infty}f(x)x^{-2-2s}\,dx
=O(1)\int_2^{\infty}\frac{1}{x(\ln x)^2}\,dx=O(1).
$$
Hence
\begin{align*}
\ep[Y_1^2(s)]=\ep[X^2(s)]
&=\sum_{p\in\mathcal{P}}\frac{1}{p^{1+2s}}
=\int_2^{\infty}\frac{1}{x^{1+2s}}\,d\pi(x)\\
&=\int_2^{\infty}\frac{1}{x^{1+2s}}\frac{dx}{\ln x}
+\int_2^{\infty}\frac{1}{x^{1+2s}}\,df(x)\\
&=\int_{2s\ln2}^{\infty}\frac{e^{-u}}{u}\,du+O(1)
=E_1(2s\ln2)+O(1)\\
&=-\log s+O(1),\qquad s\to0.
\end{align*}
\eqref{eq:lem:2-1:2} is proved.

For \eqref{eq:lem:2-1:3}, set
$u=e^{-s}\in(0,1)$, $ c=1-e^{-t}\in(0,1)$
as in \eqref{eq:proof:lem:4-1-1}. Then $e^{-s-t}=e^{-s}e^{-t}=u(1-c)$.
Thus
$$
X(e^{-s-t})-X(e^{-s})
=
\sum_{p\in\mathcal{P}}\xi_p
\left(p^{-e^{-s-t}-1/2}-p^{-e^{-s}-1/2}\right)
=
\sum_{p\in\mathcal{P}}\xi_p p^{-u-1/2}(p^{uc}-1).
$$
Because the $\xi_k$'s are independent with mean $0$ and variance $1$,
$$
\ep\left[\left(X(e^{-s-t})-X(e^{-s})\right)^2\right]
=
\sum_{p\in\mathcal{P}}p^{-1-2u}(p^{uc}-1)^2
=\int_2^{\infty}f_{u,c}(x)\,d\pi(x),
$$
where
$$
f_{u,c}(x)=x^{-1-2u}(x^{uc}-1)^2.
$$
A direct computation gives
$$
\frac{d}{dx}\ln f_{u,c}(x)
=
-\frac{1+2u}{x}
+\frac{2uc}{x(1-x^{-uc})}.
$$
Thus
$$
\left|\frac{d}{dx}\ln f_{u,c}(x)\right|\le\frac{4}{x},
\qquad x\ge2,\ 0\le u,c\le1.
$$
By the prime number theorem, there exist absolute constants $0<A\le B$ such
that for all $x\ge2$,
$$
A\frac{x}{\ln x}\le\pi(x)\le B\frac{x}{\ln x}.
$$
It follows that
\begin{align*}
\int_2^{\infty}f_{u,c}(x)\,d\pi(x)
&=f_{u,c}(x)\pi(x)\Big|_{x=2}^{\infty}
-\int_2^{\infty}\pi(x)f_{u,c}(x)
\left[\frac{d}{dx}\ln f_{u,c}(x)\right]dx\\
&\le C\int_2^{\infty}f_{u,c}(x)\frac{dx}{\ln x}
\le C\int_1^\infty
x^{-1-2u}(x^{uc}-1)^2\frac{dx}{\ln x}.
\end{align*}
Make the change of variable $y=u\log x$. Then $x=e^{y/u}$,
$dx/\ln x=e^{y/u}\,dy/y$, and
$$
x^{-1-2u}=e^{-(1+2u)y/u}=e^{-y/u-2y}.
$$
Hence
$$
\ep\left[\left(X(e^{-s-t})-X(e^{-s})\right)^2\right]
\le C\int_{0}^{\infty}
e^{-2y}(e^{cy}-1)^2\frac{dy}{y}=:CI(c).
$$
It has been shown that (cf. \eqref{eq:proof:lem:4-1-2})
$$
I(c)=t-2\ln2+2\log(1+e^{-t})\le t.
$$
This proves  \eqref{eq:lem:2-1:3}.
\end{proof}

\begin{lemma}\label{lem:4.1}
We have
$$
\mathbb E[Y_2^2(s)]=\sigma^2\log(1/s)+O(1)\qquad(s\to0),
$$
and there exists a constant $C_2>0$ such that
\begin{equation}\label{eq:lem:4.1:1}
\mathbb E\left[\left(Y_2(e^{-s-t})-Y_2(e^{-s})\right)^2\right]
\le C_2t\sigma^2
\end{equation}
for all $t>0$ and $s>0$.
\end{lemma}

\begin{proof}
Without loss of generality, we assume $\sigma=1$. Then
$$
\mathbb E[Y_2^2(s)]
=\sum_{k=1}^\infty \frac{1}{(k\log k)^{2s+1}}
=\sum_{k=1}^\infty \frac{1}{k^{1+2s}(\log k)^{1+2s}}.
$$
Since $0\le1-(\log k)^{-2s}
=1-\exp\{-s\log\log k\}\le s\log\log k$,
$$
\mathbb E[Y_2^2(s)]
=\sum_{k=1}^\infty \frac{1}{k^{1+2s}\log k}
+O\left(s\sum_{k=1}^\infty
\frac{\log\log k}{k^{1+2s}\log k}\right).
$$
The error term is $O(1)$ as $s\to0$. By the integral test,
$$
\sum_{k=1}^\infty \frac{1}{k^{1+2s}\log k}
=\int_e^\infty \frac{dx}{x^{1+2s}\log x}+O(1)
=\int_{2s}^\infty \frac{e^{-u}}{u}\,du+O(1)
=-\log s+O(1).
$$
For the second estimate, write $u=e^{-s}$, $c=1-e^{-t}$ as in
\eqref{eq:proof:lem:4-1-1}, and
$$
c_k(s)=\frac{1}{(k\log k)^{e^{-s}+1/2}}.
$$
Then $e^{-s}-e^{-s-t}=uc$, and
$$
|c_k(s+t)-c_k(s)|
=c_k(s)\left(\exp(uc\log(k\log k))-1\right)
=(k\log k)^{-u-1/2}\left((k\log k)^{uc}-1\right).
$$
It follows that
\begin{align*}
\ep\left[\left(Y_2(e^{-s-t})-Y_2(e^{-s})\right)^2\right]
&=\sum_{k\ge2}(c_k(s+t)-c_k(s))^2\\
&=\sum_{k=1}^{\infty}(k\log k)^{-2u-1}
\left((k\log k)^{uc}-1\right)^2\\
&\le 2^{2u+1}\int_2^{\infty}(x\ln x)^{-2u-1}
((x\log x)^{uc}-1)^2\,dx\\
&\le 8\int_{2\ln2}^{\infty}y^{-2u-1}(y^{uc}-1)^2
\frac{1}{\ln x+1}\,dy\\
&\le 16\int_{2\ln2}^{\infty}\frac{y^{-2u-1}}{\ln y}
(y^{uc}-1)^2\,dy\\
&\le 16\int_0^{\infty}\frac{e^{-2z}(e^{cz}-1)^2}{z}\,dz\\
&=16I(c)=16\left(t-2\ln2+2\log(1+e^{-t})\right)\le16t,
\end{align*}
by \eqref{eq:proof:lem:4-1-2}. This proves \eqref{eq:lem:4.1:1}.
\end{proof}

\begin{proof}[Proof of Theorem \ref{th:4}]
We first estimate the variances of $Y_2(s)-Y(s)$ and $Y_1(s)-Y(s)$. Let
$$
A_k(s)=p_k^{-s-1/2},\qquad
B_k(s)=(k\log k)^{-s-1/2},\qquad
C_k(s)=k^{-s-1/2}(\log k)^{-1/2}.
$$
Then
$$
C_k(s)-B_k(s)
=(k\log k)^{-1/2}e^{-s\log k}
(1-e^{-s\log\log k}).
$$
When $0\le s\le(\log k)^{-3/4}$,
$$
C_k(s)-B_k(s)
\le (k\log k)^{-1/2}(\log k)^{-3/4}\log\log k
\le Ck^{-1/2}(\log k)^{-2/3}.
$$
When $s\ge(\log k)^{-3/4}$,
$$
C_k(s)-B_k(s)
\le (k\log k)^{-1/2}e^{-(\log k)^{1/4}}
\le Ck^{-1/2}(\log k)^{-2/3}.
$$
Thus
\begin{equation}\label{eq:proof:th:3.1}
0\le C_k(s)-B_k(s)\le Ck^{-1/2}(\log k)^{-2/3}.
\end{equation}
It follows that
$$
\Var[Y_2(s)-Y(s)]
=\sigma^2\sum_{k=1}^{\infty}(B_k(s)-C_k(s))^2
\le \sigma^2C^2\sum_{k=1}^{\infty}k^{-1}(\log k)^{-4/3}
\le C\sigma^2,\qquad s\in(0,1].
$$
Next, we consider $Y_1(s)-Y(s)$. By the prime number theorem,
$$
p_k=k(\log k+\log\log k-1+o(1)).
$$
Hence
$$
\frac{p_k}{k\log k}
=1+\frac{\log\log k-1+o(1)}{\log k},
$$
and therefore
$$
\frac{\sqrt{k\log k}}{\sqrt{p_k}}
=1+O\left(\frac{\log\log k}{\log k}\right).
$$
Thus
$$
\frac1{\sqrt{p_k}}-\frac1{\sqrt{k\log k}}
=O\left(\frac{\log\log k}{k^{1/2}(\log k)^{3/2}}\right).
$$
Let $f(x)=x^{2s+1}$. Then $|f^{\prime}(x)|\le2s+1\le3$ for
$0\le x\le1$. It follows that
$$
|A_k(s)-B_k(s)|
=\left|f\left(\frac1{\sqrt{p_k}}\right)
-f\left(\frac1{\sqrt{k\log k}}\right)\right|
\le 3\left|\frac1{\sqrt{p_k}}-\frac1{\sqrt{k\log k}}\right|
\le C\frac{\log\log k}{k^{1/2}(\log k)^{3/2}}.
$$
It follows that
$$
|A_k(s)-C_k(s)|
\le |B_k(s)-C_k(s)|+|A_k(s)-B_k(s)|
\le Ck^{-1/2}(\log k)^{-2/3}.
$$
Hence,
$$
\Var[Y_1(s)-Y(s)]
=\sigma^2\sum_{k=1}^{\infty}(A_k(s)-C_k(s))^2
\le \sigma^2C^2\sum_{k=1}^{\infty}k^{-1}(\log k)^{-4/3}
\le C\sigma^2,\qquad s\in(0,1].
$$

Now we turn to the proof of \eqref{eq:th:4.1}--\eqref{eq:th:4.3}. At first,
we assume that $\{\xi_k;k\ge1\}$ are independent $N(0,1)$ random variables.
Then $Y(s)$, $Y_i(s)$ and $Y_i(s)-Y(s)$, $i=1,2$, are Gaussian processes.
Choose $s_n=e^{-t_0n}$. Then
$$
\pr\left(|Y_2(e^{-t_0n})-Y(e^{-t_0n})|\ge 2C\sqrt{\log n}\right)
\le 2\exp\left\{-\frac{4C^2\log n}
{2\Var[Y_2(e^{-t_0n})-Y(e^{-t_0n})]}\right\}
\le \exp\{-2\log n\}.
$$
It follows that
$$
Y_i(e^{-t_0n})-Y(e^{-t_0n})=O\big(\sqrt{\log n}\big)
\quad a.s.,\quad i=1,2.
$$
On the other hand, all $Y(s)$, $Y_1(s)$ and $Y_2(s)$ have the estimate as in
\eqref{eq:lem:4.1:1}. Thus, they have the sample path properties as stated in
\eqref{eq:proof:prop:2-1-1}. It follows that
$$
Y_i(e^{-T})-Y(e^{-T})=O\big(\sqrt{\log T}\big)\quad a.s.
\text{ as }T\to\infty,\ i=1,2.
$$
Thus,
$$
Y_i(s)-Y(s)=O\big(\sqrt{\log\log(1/s)}\big)\quad a.s.
\text{ as }s\to0,\ i=1,2,
$$
which also implies that
$$
\max_{s\le t\le1}|Y_i(t)-Y(t)|
=O\big(\sqrt{\log\log(1/s)}\big)\quad\text{in }L_2
\text{ as }s\to0,\ i=1,2,
$$
since $Y_i(t)-Y(t)$, $i=1,2$, are Gaussian (cf. Corollary 3.2 of Ledoux
and Talagrand \cite{LT1991}). Hence, \eqref{eq:th:4.1}--\eqref{eq:th:4.3} are
proven when $\{\xi_k;k\ge1\}$ are normal.

At last, we consider the general case that $\{\xi,\xi_k;k\ge1\}$ are
independent and identically distributed random variables with mean $0$ and
finite variance $\sigma^2$. Let
$$
d_k(s)=p_k^{-s-1/2},\quad (k\log k)^{-s-1/2},\quad\text{or}\quad
k^{-s-1/2}(\log k)^{-1/2}.
$$
With the same arguments as showing Theorem \ref{th:1}, based on
$\{\xi_k;k\ge1\}$, we find a sequence of independent $N(0,1)$ random
variables $\{\eta_k;k\ge1\}$ (the construction of normal random variables
does not depend on the weights $d_k(s)$) such that
\begin{align*}
&\sum_{k=1}^{\infty}(\xi_k-\sigma\eta_k)d_k(s)
=o\Big(\sqrt{\log(1/s)\log\log\log(1/s)}\Big)\quad a.s.,\\
&\sup_{s\le t\le1}\left|
\sum_{k=1}^{\infty}(\xi_k-\sigma\eta_k)d_k(t)
\right|
=o\Big(\sqrt{\log(1/s)}\Big)\quad\text{in }L_2,
\end{align*}
and
$$
\sum_{k=1}^{\infty}(\xi_k-\sigma\eta_k)d_k(s)
=o\Big(\sqrt{\log\log(1/s)}\Big)\quad a.s.
$$
if \eqref{eq:th:1.3} holds, for
$d_k(s)=p_k^{-s-1/2}$, $(k\log k)^{-s-1/2}$ and
$k^{-s-1/2}(\log k)^{-1/2}$. This proves the theorem.
\end{proof}

\section{Random Dirichlet series: noncritical case}\label{sec:noncritical}
\setcounter{equation}{0}

In this section, we give the proofs of Theorem \ref{th:3}, Corollaries
\ref{cor:4} and \ref{cor:5}. Let
$$
Y(s)=\sum_{k\ge2}\frac{(\log k)^{\alpha}}{k^{s+1/2}}\xi_k,
\qquad \alpha>-1/2.
$$
In this section, we show the limit theorems for the case $\alpha>-1/2$.

\begin{proof}[Proof of Theorem \ref{th:3}]
Since the proof is similar to that of Theorem \ref{th:1}, we only give the
proof of \eqref{eq:th:3.3}. Under condition \eqref{eq:th:1.3}, there exists a
standard Brownian motion $W(t)$ such that \eqref{eq:proof:th1.0} holds.

Write $\eta_k=W(k)-W(k-1)$,
$f_s(k)=(\log k)^{\alpha}k^{-s-1/2}$,
$h(x)=x^{1/2}(\log x)^{-1/2-\epsilon/2}$. Then $\{\eta_k;k\ge1\}$ is a
sequence of independent $N(0,1)$ random variables,
$|h(k)-h(k-1)|\le Ck^{-1/2}(\log k)^{-1/2-\epsilon/2}$. Similar to
\eqref{eq:proof:th2.1}, by Abel summation,
\begin{align*}
&\left|\sum_{k=1}^{\infty}(\xi_k-\sigma\eta_k)f_s(k)\right|
=
\left|\sum_{k=1}^{\infty}(f_s(k)-f_s(k-1))[S_k-\sigma W(k)]\right|\\
&\overset{a.s.}{=}
o(1)\sum_{k=1}^{\infty}(f_s(k)-f_s(k-1))h(k)+O(1)\\
&=
o(1)\sum_{k=1}^{\infty}(h(k)-h(k-1))f_s(k)+O(1)\\
&=
o(1)\sum_{k=1}^{\infty}
\frac{k^{-1/2}(\log k)^{-1/2-\epsilon/2}(\log k)^{\alpha}}
{k^{s+1/2}}+O(1)\\
&=o(1)\int_e^{\infty}
\frac{(\log x)^{\alpha-1/2-\epsilon/2}}{x^{s+1}}\,dx+O(1)\\
&=o(1)\int_1^{\infty}y^{\alpha-1/2-\epsilon/2}e^{-sy}\,dy+O(1)\\
&=\begin{cases}
o(s^{-\alpha-1/2+\epsilon/2}), & \alpha>-1/2+\epsilon/2,\\
o(\log(1/s)), & \alpha=-1/2+\epsilon/2,\\
O(1), &\alpha<-1/2+\epsilon/2,
\end{cases}
\end{align*}
by Lemma \ref{lem:5-1}. Hence,
$$
s^{\alpha+1/2}\left|
\sum_{k=1}^{\infty}(\xi_k-\sigma\eta_k)f_s(k)
\right|
=o(s^{\tau})\quad a.s.
\text{ for }\tau<\frac{\epsilon}{2}\wedge\left(\alpha+\frac12\right).
$$
By Proposition \ref{prop:2-2}, \eqref{eq:th:3.3} is proven.
\end{proof}

\begin{proof}[Proof of Corollary \ref{cor:4}]
\eqref{eq:cor:4.1} follows from \eqref{eq:th:3.2} immediately.
\eqref{eq:cor:4.4} follows from \eqref{eq:cor:4.1} by noting that
$$
\Var\left(\int_0^{\infty}y^{\alpha}e^{-ty}\,dB(y)\right)
=\int_0^{\infty}y^{2\alpha}e^{-2ty}\,dy
=\frac{\Gamma(2\alpha+1)}{(2t)^{2\alpha+1}}.
$$
To show that \eqref{eq:cor:4.5} is implied by \eqref{eq:cor:4.2}, it is
sufficient to notice that
$$
\sup\{g(1):g\in\mathscr{K}_{\alpha}\}
=2^{-\alpha-1/2}\sqrt{\Gamma(2\alpha+1)}.
$$
In fact, for any $f\in\mathscr{K}_B$,
$$
g(1)=\int_0^{\infty}y^{\alpha}e^{-y}f^{\prime}(y)\,dy
\le \left(\int_0^{\infty}y^{2\alpha}e^{-2y}\,dy\right)^{1/2}
\left(\int_0^{\infty}[f^{\prime}(y)]^2\,dy\right)^{1/2}
\le\left(\frac{\Gamma(2\alpha+1)}{2^{2\alpha+1}}\right)^{1/2},
$$
and for
$$
f^{\prime}(x)=
\frac{y^{\alpha}e^{-y}}
{\left(\int_0^{\infty}y^{2\alpha}e^{-2y}\,dy\right)^{1/2}},
$$
we have $\int_0^{\infty}[f^{\prime}(y)]^2\,dy=1$ and
$$
g(1)=\int_0^{\infty}y^{\alpha}e^{-y}f^{\prime}(y)\,dy
=\left(\frac{\Gamma(2\alpha+1)}{2^{2\alpha+1}}\right)^{1/2}.
$$
To show \eqref{eq:cor:4.2}, denote
$$
\widetilde{X}(s)=\int_0^{\infty}y^{\alpha}e^{-sy}\,dB(y),
\qquad
\widetilde{X}_s(t)=
\frac{s^{\alpha+1/2}\widetilde{X}(st)}
{\sqrt{2\log\log(1/s)}}.
$$
By \eqref{eq:th:3.1}, for any $t_0>0$,
$$
\sup_{t\ge t_0}\left|
\frac{s^{\alpha+1/2}Y(st)}{\sqrt{\log\log(1/s)}}
-\widetilde{X}_s(t)
\right|\to0\quad a.s.\text{ as }s\to0.
$$
Thus, it is sufficient to show that the set of limit points of
$\widetilde{X}_s(\cdot)$ is $\mathscr{K}_{\alpha}$. Notice that
$$
\widetilde{X}(s)
=\int_0^{\infty}B(y)(sy-\alpha)y^{\alpha-1}e^{-sy}\,dy
=s^{-\alpha-1/2}\int_0^{\infty}
\frac{B(y/s)}{s^{-1/2}}(y-\alpha)y^{\alpha-1}e^{-y}\,dy.
$$
Thus
$$
\widetilde{X}_s(t)
=\int_0^{\infty}t^{-\alpha}
\frac{B(y/(ts))}{\sqrt{2(1/s)\log\log(1/s)}}
(y-\alpha)y^{\alpha-1}e^{-y}\,dy.
$$
By Strassen's functional law of the iterated logarithm for a Brownian
motion, the set of limit points of
$$
\frac{B(\cdot/(ts))}{\sqrt{2(1/s)\log\log(1/s)}}
$$
is $\mathscr{K}_B$. So the set of limit points of $\widetilde{X}_s(\cdot)$
is
$$
\left\{
\left(t^{-\alpha}\int_0^{\infty}
f\left(\frac{y}{t}\right)(y-\alpha)y^{\alpha-1}e^{-y}\,dy\right)_{t>0};
\ f\in\mathscr{K}_B
\right\}.
$$
It is easily seen that
$$
t^{-\alpha}\int_0^{\infty}
f\left(\frac{y}{t}\right)(y-\alpha)y^{\alpha-1}e^{-y}\,dy
=\int_0^{\infty}y^{\alpha}e^{-ty}\,df(y).
$$
The proof is completed.
\end{proof}

\begin{proof}[Proof of Corollary \ref{cor:5}]
Let $\alpha>-1/2$ and
\begin{equation}\label{eq:proof:cor:5.1}
G(t)=\widetilde{X}(1/t)=\int_0^{\infty}y^{\alpha}e^{-y/t}\,dB(y),
\qquad t>0.
\end{equation}
Then $G(t)$ is a self-similar Gaussian process with index
$H=\alpha+1/2$, and
\begin{equation}\label{eq:proof:cor:5.2}
\ep[G(t)G(s)]=\Gamma(2\alpha+1)
\left(\frac{st}{s+t}\right)^{2\alpha+1}.
\end{equation}
It follows that $U(t)=G(e^t)e^{-t(\alpha+1/2)}$ is a stationary Gaussian
process with
$$
\gamma(h)=\ep[U(t)U(t+h)]
=\Gamma(2\alpha+1)
\left(\frac{1}{1+e^{-|h|}}\right)^{2\alpha+1}
e^{-|h|(\alpha+1/2)}.
$$
$\gamma(h)$ satisfies the conditions of Pickands
\cite[Theorem 5.4]{Pickans1967}:
$$
\exists\epsilon>0,\quad
\limsup_{|h|\to0}|h|^{-\epsilon}\{\gamma(0)-\gamma(h)\}<\infty,
\qquad
\lim_{|h|\to\infty}\gamma(h)\log|h|=0.
$$
Thus it follows that
$$
\sup_{0\le v\le t}U(v)-\sqrt{2\gamma(0)\log t}\to0
\quad a.s.\text{ as }t\to\infty.
$$
Hence
\begin{align*}
&\sup_{s\le t\le1}t^{\alpha+1/2}G(1/t)
-\sqrt{2\gamma(0)\log\log(1/s)}\\
&=\sup_{0\le v\le\log(1/s)}U(v)
-\sqrt{2\gamma(0)\log\log(1/s)}\to0
\quad a.s.\text{ as }s\to0.
\end{align*}
Then \eqref{eq:cor:5.1} and \eqref{eq:cor:5.2} follow from
\eqref{eq:th:3.1} and \eqref{eq:th:3.3}, respectively.
\end{proof}

\section{Remarks}

\emph{(a) A relationship between the geometrically weighted random series
and the random Dirichlet series.}
Zhang \cite{Zhang1997} considered the Gaussian approximation of the random
geometric random series $\sum_{k=1}^{\infty}\beta^k\xi_k$ as
$\beta\nearrow1$. If we write $\beta=e^{-s}$, then
$$
s=\frac12\log(1/\beta^2)
=\frac12(1-\beta^2)+O\bigl((1-\beta^2)^2\bigr),
$$
$$
\sum_{k=1}^{\infty}\beta^k\xi_k
=\sigma\sum_{k=1}^{\infty}e^{-sk}\bigl[W(k)-W(k-1)\bigr]+O(A(s))
=\sigma\int_0^{\infty}e^{-sy}\,dW(y)+O(A(s))\quad a.s.,
$$
with the convergence rate $A(s)$ given in Zhang \cite{Zhang1997}. The limit
process is the Laplace transform of $B(\cdot)$, the same as that of a random
Dirichlet series with $\alpha=0$. As a result, the asymptotic behavior of the
geometrically weighted random series is the same as that of a random
Dirichlet series with $\alpha=0$; for example,
$$
\left(\sqrt{1-\beta^2}\sum_{k=1}^{\infty}\beta^{tk}\xi_k\right)_{t>0}
\implies
\left(\sigma\sqrt{2}\int_0^{\infty}e^{-ty}\,dB(y)\right)_{t>0}.
$$
See also Gao et al. \cite{GGX2024}.

\emph{(b) Small ball probabilities and Chung's LILs.}
The Gaussian process $G(t)$ defined as in \eqref{eq:proof:cor:5.1} has the
covariance function given in \eqref{eq:proof:cor:5.2}. The small ball
probabilities of this kind of processes were obtained by Aurzada et al.
\cite{Aurzada2013}.

\begin{lemma}\label{lem:smallball}
There is a constant $\kappa_{\alpha}>0$ such that, as $\epsilon\to0$,
$$
-\ln\pr\left(\int_0^1G^2(t)\,dt
\le \epsilon^2\frac{\Gamma(2\alpha+1)}{2^{2\alpha+1}}\right)
\sim \frac{1}{3(\alpha+1)\pi^2}|\log\epsilon|^3,
$$
\begin{align}\label{eq:lem:smallball:2}
\kappa_{\alpha}|\log\epsilon|^3
&\lesssim
-\ln\pr\left(\sup_{0\le t\le1}|G(t)|
\le \epsilon\frac{\sqrt{\Gamma(2\alpha+1)}}{2^{\alpha+1/2}}\right)
\nonumber\\
&\lesssim
\frac{1}{3(\alpha+1)\pi^2}|\log\epsilon|^3.
\end{align}
\end{lemma}

$G(s)$ and $X(s)$ do not have Chung's LIL of the traditional form. From the
small ball estimates provided in \eqref{eq:lem:smallball:2}, along with a
rescaling argument, the Borel--Cantelli lemma and the zero-one law, we can
obtain
$$
\liminf_{T\to\infty}
\frac{\log\left\{T^{-\alpha-1/2}\sup_{0<t\le T}|G(t)|\right\}}
{(\log\log T)^{1/3}}
=-c_{\alpha}^{-1/3}\quad a.s.
$$
for some positive constant
$c_{\alpha}\in[\kappa_{\alpha},\frac{1}{3(\alpha+1)\pi^2}]$, which together
with \eqref{eq:th:3.3} yields
\begin{align}\label{eq:rk:2}
\liminf_{s\to0}
\frac{\log\left\{s^{\alpha+1/2}\sup_{t\ge s}|Y(t)|\right\}}
{(\log\log(1/s))^{1/3}}
=-c_{\alpha}^{-1/3}\quad a.s.,
\end{align}
if \eqref{eq:th:1.3} is satisfied.

A very recent manuscript of Xia \cite{Xia2026} has shown that
$$
\mathfrak{F}(\epsilon):=
-\ln\pr\left(\sup_{0\le t\le1}|G(t)|
\le \epsilon\frac{\sqrt{\Gamma(2\alpha+1)}}{2^{\alpha+1/2}}\right)
\sim \frac{1}{3\pi^2(2\alpha+1)}|\ln\epsilon|^3,
\quad\text{as }\epsilon\to0,
$$
and \eqref{eq:rk:2} holds with
$c_{\alpha}=\frac{1}{3\pi^2(2\alpha+1)}$ under the assumption of a finite
second moment.

To obtain Chung's LIL of the form \eqref{eq:cor:3.1}, we need a more precise
estimate of the small ball probability. For the special case $\alpha=0$,
Letwin and Sawhney \cite{LS2026} obtained a precise estimate as
$$
\mathfrak{F}(\epsilon)
=-\frac{1}{3\pi^2}|\ln\epsilon|^3
+O\left(|\log\epsilon|^{5/2}\sqrt{\log\log(1/\epsilon)}\right).
$$
It is expected that the power $5/2$ can be reduced to $2$, and
\begin{align}\label{eq:conjecture:1}
\mathfrak{F}(e^{-L})=
\frac{1}{3\pi^2(2\alpha+1)}
\left(L^3+c_{1,\alpha}L^2\log L+c_{2,\alpha}L^2\right)
+o(L^2),\;\; \text{ as } L\to \infty.
\end{align}
  Then for
$\beta(T)=(3\pi^2(2\alpha+1))^{1/3}T^{1/3}
-\frac{c_{1,\alpha}}{9}\log T$, we have
\begin{align}\label{eq:conjecture:2}
\mathfrak{F}(xe^{-\beta(T)})
&=\mathfrak{F}(e^{-\{\beta(T)-\log x\}})=T+\mathfrak{C}(x)T^{2/3}+ o(T^{2/3}),\\
\mathfrak{C}(x)&=
 3\frac{c_{2,\alpha}/3+
\{c_{1,\alpha}\log(3\pi^2(2\alpha+1))\}/9-\log x}
{\{3\pi^2(2\alpha+1)\}^{1/3}}.\nonumber
\end{align}
Thus, we have the following conjecture.

\begin{proclaim}{Conjecture}
Let $c_{\alpha}=\frac{1}{3\pi^2(2\alpha+1)}$. Then
\begin{align}\label{eq:conjecture:3}
\liminf_{T\to\infty}
\frac{\exp\{c_{\alpha}^{-1/3}(\log\log T)^{1/3}\}}
{T^{\alpha+1/2}(\log\log T)^{c_{1,\alpha}/9}}
\sup_{0<t\le T}|G(t)|
=\frac{\sqrt{\Gamma(2\alpha+1)}}{2^{\alpha+1/2}}\lambda_{\alpha}
\quad a.s.,
\end{align}
\begin{align*}
\liminf_{s\to0}
\frac{\exp\{c_{\alpha}^{-1/3}(\log\log(1/s))^{1/3}\}
s^{\alpha+1/2}}
{(\log\log(1/s))^{c_{1,\alpha}/9}}
\sup_{t\ge s}|Y(t)|
=\sigma\frac{\sqrt{\Gamma(2\alpha+1)}}{2^{\alpha+1/2}}\lambda_{\alpha}
\quad a.s.,
\end{align*}
where $\lambda_{\alpha}$ is the solution of
$$
c_{2,\alpha}/3+
\{c_{1,\alpha}\log(3\pi^2(2\alpha+1))\}/9-\log x=0.
$$
\end{proclaim}

\begin{proof}[Proof of \eqref{eq:conjecture:3} from \eqref{eq:conjecture:1}] Let 
$$\gamma(T)
=\frac{\exp\{c_{\alpha}^{-1/3}(\log\log T)^{1/3}\}}
{(\log\log T)^{c_{1,\alpha}/9}},$$
$$ Y(T)=\gamma(T)T^{-\alpha-1/2}\sup_{0<t\le t\le T}|G(t)| \frac{2^{\alpha+1/2}}{\sqrt{\Gamma(2\alpha+1)}}. $$
Then
\begin{align*}
& -\ln \pr\left(Y(T)\le x\right) 
= 
\mathfrak{F}(xe^{-\beta(\log\log T)}) \\
=&\log\log T+\mathfrak{C}(x)(\log\log T)^{2/3}+ o\big((\log\log T)^{2/3}\big). 
\end{align*}
For $x<\lambda_{\alpha}$, $\mathfrak{C}(x)>0$. Choose $T_k=e^{k/(\log k)}$. Then for $k_0$ large enough,
\begin{align*}
&- \ln \pr\left(Y(T_k)\le x\right) 
=   \log k-\log\log k+\mathfrak{C}(x)(\log k)^{2/3}+o\big((\log k)^{2/3}\big)\\
\ge & \log k+ \mathfrak{C}(x)(\log k)^{2/3}/2, \;\; k\ge k_0.
\end{align*}
It follows that
$$\sum_{k=k_0}^{\infty}\pr\left(Y(T_k)\le x\right)\le \sum_{k=k_0}^{\infty} \exp\{-\log k-\mathfrak{C}(x)(\log k)^{2/3}/2\}<\infty, $$
which implies
$$ \liminf_{k\to \infty} Y(T_k)>x\;\; a.s. $$
For $T_k\le T\le T_{k+1}$, by noting that $\gamma(T)$ is a non-decreasing function of $T$ when $T$ is large enough, we have
$$ Y(T)\ge \left(T_k/T_{k+1}\right)^{\alpha+1/2}Y(T_k) $$
for $k$ large enough. Since $T_k/T_{k+1}\to 1$, it follows that
$$ \liminf_{k\to \infty} Y(T)\ge \lambda_{\alpha}\;\; a.s. $$
For showing  the inequality in the opposite direction, we let $T_k=e^{k(\log k)^4}$. Then for any $x>\lambda_{\alpha}$, we have $\mathfrak{C}(x)<0$. Thus, for $k_0$ large enough,
\begin{align*}
 - \ln \pr\left(Y(T_k)\le x\right) 
=  & \log k+2\log\log k+\mathfrak{C}(x)(\log k)^{2/3}+o\big((\log k)^{2/3}\big)\\
\le & \log k+ \mathfrak{C}(x)(\log k)^{2/3}/2\le \log k,\;\; k\ge k_0.
\end{align*}
It follows that
$$\sum_{k=1}^{\infty}\pr\left(Y(T_k)\le x\right)\ge \sum_{k=k_0}^{\infty}\frac{1}{k}=\infty. $$
However, we can not apply the Borel-Cantelli lemma directly since $\{Y(T_k)\}$ are not independent. For considering the independence, we let $h(T)=T(\log T)^4$, 
\begin{align*}
G_k(t)=&\frac{2^{\alpha+1/2}}{\sqrt{\Gamma(2\alpha+1)}}\int_{y\in \big(h(T_{k-1}),h(T_k)\big)}y^{\alpha}e^{-y/t}\,dB(y),\\
\overline{G}_k(t,u)=&\frac{2^{\alpha+1/2}}{\sqrt{\Gamma(2\alpha+1)}}
\int_{y\not\in\big(h(T_{k-1}),h(T_k)\big), 0\le y\le u}y^{\alpha}e^{-y/t}\,dB(y),\\
\overline{G}_k(t)=&\overline{G}_k(t,\infty),
\end{align*}
$t,u>0$. Then $G_k(t)$ and $\overline{G}_k(t)$ are independent Gaussian process with
$G_k(t)+\overline{G}_k(t)=G(t)\frac{2^{\alpha+1/2}}{\sqrt{\Gamma(2\alpha+1)}}$. 
By the Anderson inequality, we have
$$\sum_{k=1}^{\infty}\pr\left(\gamma(T_k)T_k^{-\alpha-1/2}\sup_{0<t\le T_k}|G_k(t)|\le x\right)
\ge \sum_{k=1}^{\infty}\pr(Y(T_k)\le x)=\infty. $$
Now, $\{G_k(t); k=1,2\ldots,\}$ are independent. By the Borel-Cantelli lemma,
$$\liminf_{k\to \infty}\gamma(T_k)T_k^{-\alpha-1/2}\sup_{0<t\le T_k}|G_k(t)|\ge x\;\; a.s. $$
It is remained to show that
\begin{align}\label{eq:proof:conjecure:2}\liminf_{k\to \infty}\gamma(T_k)T_k^{-\alpha-1/2}\sup_{0<t\le T_k}|\overline{G}_k(t)|=0\;\; a.s. 
\end{align}

Similar to \eqref{eq:proof:lem3.1}, 
$$\sup_{0<t\le T_k}|\overline{G}_k(t)|=\sup_{u>0}|\overline{G}(T_k,u)|.$$
Note that $\overline{G}_k(T_k,u), u>0$, is a symmetric Gaussian process with independent increments. It follows that
\begin{align*}
&\pr\left(\gamma(T_k)T_k^{-\alpha-1/2}\sup_{0<t\le T_k}|\overline{G}_k(t)|\ge \epsilon\right)
=\pr\left(\gamma(T_k)T_k^{-\alpha-1/2}\sup_{u>0}|\overline{G}_k(T_k,u)|\ge \epsilon\right)\\
\le & 2\pr\left(\gamma(T_k)T_k^{-\alpha-1/2}|\overline{G}_k(T_k)|\ge \epsilon\right)
\le 4\exp\left\{ -\frac{\epsilon^2}{2\gamma^2(T_k)T_k^{-2\alpha-1}\sigma_k^2}\right\},
\end{align*}
where
\begin{align*}
\sigma_k^2=& \Var\{\overline{G}_k(T_k)\}
=\frac{2^{2\alpha+1}}{\Gamma(2\alpha+1)}\int_{x\not\in\big(h(T_{k-1}),\, h(T_k)\big)} y^{2\alpha}e^{-2y/T_k}
dy\\
=&\frac{T_k^{2\alpha+1}}{\Gamma(2\alpha+1)}\left[\int_0^{h(T_{k-1})/T_k} +\int_{ h(T_k)/T_k}^{\infty}\right]  y^{2\alpha}e^{-y}dy=:T_k^{2\alpha+1}\delta_k^2.
\end{align*}
It is easily seen that
$$\delta_k^2\asymp \big(h(T_{k-1})/T_k\big)^{2\alpha+1}+\exp\left\{ -h(T_k)/T_k\right\}\big(h(T_k)/T_k\big)^{2\alpha}
\le  \exp\{-2c(\log k)^4\}$$
for some $c>0$, and $\gamma(T_k)=\exp\{c_{\alpha}^{-1/3}(\log k)(1+o(1))\}$. It follows that
$$\pr\left(\gamma(T_k)T_k^{-\alpha-1/2}\sup_{0<t\le T_k}|\overline{G}_k(t)|\ge \epsilon\right)
\le \exp\{-c\exp\{c(\log k)^4\}\}.$$
Hence
$$\sum_{k=1}^{\infty}\pr\left(\gamma(T_k)T_k^{-\alpha-1/2}\sup_{0<t\le T_k}|\overline{G}_k(t)|\ge \epsilon\right)<\infty. $$
This proves \eqref{eq:proof:conjecure:2}. 
\end{proof} 

\section*{Statements \& Declarations}

\noindent{\bf Funding.} This work is supported by grants from the National
Key R\&D Program of China (Grant No. 2024YFA1013502), the National Natural
Science Foundation of China (Grant No. U23A2064), and the Summit Advancement
Disciplines of Zhejiang Province (Zhejiang Gongshang University --
Statistics).

\bigskip
\noindent{\bf Competing Interests.} The author has no relevant financial or
non-financial interests to disclose.

\bigskip
\noindent{\bf Data Availability.} The manuscript has no associated data.

\bigskip
\noindent{\bf Author Contributions.} The single author contributed to the
derivation of the results, the manuscript preparation, and the final
manuscript. The rapid derivation of the expansion of the $k$-th prime and of the covariances given in Lemma \ref{lem:2-1} was aided by DeepSeek.

\end{document}